\documentclass[11pt]{amsart}
\usepackage[margin=1in]{geometry}
\usepackage{microtype}
\usepackage{amssymb,amsmath,amsthm,mathrsfs,mathtools}
\usepackage{tikz-cd}
\usepackage{hyperref}
\hypersetup{breaklinks, colorlinks}

\newtheorem{thm}{Theorem}[section]
\newtheorem{lemma}[thm]{Lemma}
\newtheorem{cor}[thm]{Corollary}
\newtheorem{prop}[thm]{Proposition}

\theoremstyle{definition}
\newtheorem{defn}[thm]{Definition}
\newtheorem{rem}[thm]{Remark}
\newtheorem{ex}[thm]{Example}

\numberwithin{equation}{section}

\def\D{\mathbf{D}^\mathrm{b}}
\def\DD{\mathbf{D}}

\def\P{\mathbb{P}}
\def\LL{\mathbb{L}}

\def\B{\mathcal{B}}
\def\E{\mathcal{E}}
\def\O{\mathcal{O}}
\def\Z{\mathcal{Z}}

\def\cA{\mathcal{A}}
\def\cC{\mathcal{C}}
\def\cF{\mathcal{F}}
\def\cG{\mathcal{G}}
\def\cH{\mathcal{H}}
\def\cI{\mathcal{I}}
\def\cL{\mathcal{L}}
\def\cQ{\mathcal{Q}}
\def\cN{\mathcal{N}}
\def\cS{\mathcal{S}}
\def\cT{\mathcal{T}}
\def\cV{\mathcal{V}}

\def\sL{\mathscr{L}}

\newcommand{\kk}{\Bbbk}
\newcommand{\wt}{\widetilde}

\newcommand{\rank}{\operatorname{rank}}
\newcommand{\perf}{\operatorname{perf}}
\newcommand{\sg}{\operatorname{sg}}
\newcommand{\Bl}{\operatorname{Bl}}
\newcommand{\M}{\operatorname{M}}
\newcommand{\Gr}{\operatorname{Gr}}
\newcommand{\Proj}{\operatorname{Proj}}
\newcommand{\Spec}{\operatorname{Spec}}
\newcommand{\Coh}{\operatorname{Coh}}
\newcommand{\QCoh}{\operatorname{QCoh}}
\newcommand{\op}{\operatorname{op}}
\newcommand{\End}{\operatorname{End}}
\newcommand{\Ext}{\operatorname{Ext}}
\newcommand{\Hom}{\operatorname{Hom}}
\newcommand{\RHom}{\operatorname{RHom}}
\newcommand{\coker}{\operatorname{coker}}
\newcommand{\codim}{\operatorname{codim}}
\newcommand{\Sym}{\operatorname{Sym}}
\newcommand{\id}{\operatorname{id}}
\newcommand{\Bs}{\operatorname{Bs}}

\newcommand{\sEnd}{\mathscr{E}\kern -1pt nd}
\newcommand{\sHom}{\mathscr{H}\kern -2pt om}

\title{Nodal quintic del Pezzo threefolds and their derived categories}
\author{Fei Xie}
\address{School of Mathematics, University of Edinburgh, James Clerk Maxwell Building, Peter Guthrie Tait Road, Edinburgh, EH9 3FD, UK}
\email{fei.xie@ed.ac.uk}

\begin{document}
\maketitle
\begin{abstract}
    We construct a Kawamata type semiorthogondal decomposition for the bounded derived category of coherent sheaves of nodal quintic del Pezzo threefolds, decomposing the bounded derived category into bounded derived categories of finite dimensional algebras. This is achieved by constructing birational maps from nodal quintic del Pezzo threefolds to quadric surface fibrations over the projective line.
\end{abstract}
\section{Introduction}
The most interesting smooth Fano threefolds are the ones with Picard rank $1$. Their derived categories are quite well known \cite{kuz3f}. But much less is known about the derived categories of singular Fano threefolds. The first examples are given by Kawamata \cite{kawaodp} for nodal del Pezzo threefolds of degree $6$ and $8$ (more details below). In both examples the threefolds have only one ordinary double point (or nodal point). It is an interesting and important question to understand the derived categories of other singular Fano threefolds and see what they look like when the Fano threefolds have multiple nodal points. In this paper we study the bounded derived category of coherent sheaves $\D(X)$ for nodal quintic del Pezzo threefolds $X$ (terminal Gorenstein Fano threefolds of index $2$ and degree $5$) over an algebraically closed field $\kk$ of characteristic $0$ and construct a \textit{Kawamata type semiorthogonal decomposition} (KSOD) defined in \cite[Definition 4.1]{kpsk-1} for Gorenstein projective varieties. This answers the question asked in \cite[Example 4.16]{kpsk-1} about the existence of KSODs for these threefolds. A KSOD is an \textit{admissible} semiorthogonal decomposition (SOD) of the type
\begin{equation}\label{ksod}
    \D(X)=\langle \cA, \D(S_1), \dots, \D(S_n) \rangle
\end{equation}
where $\cA$ is a subcategory of $\DD^{\perf}(X)$, $S_1, \dots, S_n$ are finite dimensional $\kk$-algebras and $\D(S_i), 1\leqslant i\leqslant n$ are bounded derived categories of finitely generated right modules over $S_i$. Here admissible means that the components of the SOD are admissible subcategories.

The classification of terminal Gorenstein del Pezzo threefolds of degree $5$ states that they have only nodal singularities and the number of nodal points is at most $3$. Moreover, the quintic del Pezzo threefold $X_m$ with $m$ nodes ($0\leqslant m\leqslant 3$) exists and is unique up to isomorphism for each $m$; see \cite[Corollary 8.7]{prodp3fold}. The SOD of the derived category $\D(X_0)$ of the smooth quintic del Pezzo threefold is well known; see \cite[Theorem 4.2]{kuz3f}. It has a full exceptional collection of four objects. In this paper, we will focus on the singular threefolds $X_m, 1\leqslant m\leqslant 3$ and show that two of the exceptional objects in the smooth case are replaced by derived categories of finite dimensional algebras. Furthermore, we will see that the construction of the SODs of $\D(X_m)$ is closely related to derived categories of a chain of $m$ $\P^1$'s. For a chain of $\P^1$'s, there is a SOD as below.

\begin{prop}[{\cite{burratsingc}}]\label{chainp1}
    Let $\Gamma$ be a chain of $n+1$ $\P^1$'s. Then $\D(\Gamma) =\langle \D(R_n), \O_{\Gamma} \rangle$. Here $R_0=\kk$ and $R_n, n\geqslant 1$ is the path algebra of quiver with relations
    \begin{equation}\label{rn}
        R_n= \kk \left\{\left.
        \begin{tikzcd}
            \bullet \arrow[bend left]{r}{\alpha_1} & \bullet \arrow[bend left]{l}{\beta_1} \arrow[bend left]{r}{\alpha_2} & \bullet \arrow[bend left]{l}{\beta_2} \arrow[draw=none, yshift=-0.5ex]{r}{\dots} & \bullet \arrow[bend left]{r}{\alpha_n} & \bullet \arrow[bend left]{l}{\beta_n}
        \end{tikzcd}
        \right| 
        \alpha_i\beta_i=0,\, \beta_i\alpha_i=0,\, \forall1\leqslant i\leqslant n
        \right\}.
    \end{equation}
\end{prop}

The main result of the paper is the construction of the following Kawamata type SODs, which indicates that $\D(X_m)$ behave in a similar way as $\D(\Gamma)$.

{\renewcommand{\thethm}{\ref{mainthm}}
\begin{thm}
    Let $X_m$ be the quintic del Pezzo threefolds with $m$ nodes for $m=1,2,3$. Then there is an admissible semiorthogonal decomposition
\begin{equation*}
    \D(X_m)=\langle \D(R_{m-1}), \D(R_1), \O_{X_m}, \O_{X_m}(1)\rangle
\end{equation*}
where $R_0=\kk$ and $R_1, R_2$ are the path algebras defined by (\ref{rn}).
\end{thm}
}

The reason we care about Kawamata type SODs for a singular scheme $X$ is because it gives a complete decomposition of the ``singular part" of $\D(X)$ in terms of finite dimensional algebras. More precisely, the singularity of $\D(X)$ is measured by the Orlov's triangulated categories $\DD_{\sg}(X)$ of singularities, which is defined as the Verdier quotient of $\D(X)$ by the subcategory $\DD^{\perf}(X)$ of perfect complexes. In this sense Kawamata type SOD would induce a SOD of $\DD_{\sg}(X)$ by triangulated categories of singularities of finite dimensional algebras. 

If a Kawamata type SOD does exist, one can further ask whether there exists a full exceptional collection for the subcategory $\cA$ in the SOD (\ref{ksod}). In the main result the answer is yes with $\cA=\langle \O_{X_m}, \O_{X_m}(1)\rangle$. The known examples where $\cA$ has a full exceptional collection is provided in \cite[\S4]{kpsk-1}. In dimension $1$, a chain of projective lines studied by Burban (see Proposition \ref{chainp1}) is one family of such examples. In dimension $2$, Karmazyn-Kuznetsov-Shinder \cite{kks-surf} proved that a projective Gorenstein toric surface has a SOD of this kind if and only if it has the trivial Brauer group. Moreover, their method also provides some non-toric examples: the du Val sextic del Pezzo surfaces \cite{kuzdp6} and the du Val quintic del Pezzo surfaces \cite{xiedp5}. In dimension $3$, Kawamata \cite{kawaodp} provides two such examples, the nodal quadric threefold and the nodal sextic del Pezzo threefold, by investigating the derived category of a threefold with an ordinary double point. More explicitly, Kawamata proved in \cite[Theorem 6.1 and 5.1]{kawaodp} that for certain threefolds $X$ with an ordinary double point, the right (or left) orthogonal complement to $\cA$ in $\D(X)$ is equivalent to $\D(R_1)$ (see (\ref{rn}) for the definition of $R_1$). One notes that $\D(R_1)$ is also the right orthogonal complement to the structure sheaf in the derived category of a chain of two $\P^1$'s by Burban. This phenomena is expected by Kn\"{o}rrer periodicity \cite[Theorem 2.1]{orlsing}. 

For a threefold $X$ with multiple nodal points, one naturally wonders when the SOD (\ref{ksod}) of $\D(X)$ exists and if it does exist, which algebras $S_i$ will appear. It is already known by \cite[Example 3.15 and 4.16]{kpsk-1} (see also \cite[Corollary 3.7]{psksoddp}) that derived categories of nodal del Pezzo threefolds of degree between $1$ and $4$ do not have Kawamata type SODs. The main result of paper confirms that Kawamata type SODs do exist for degree $5$. This is detailed in \S \ref{derivedcat} and it uses a generalization of Kawamata's result \cite[Theorem 6.1 and 5.1]{kawaodp} mentioned before, which is explained in Proposition \ref{kawaprop}, to produce a semiorthogonal component $\D(R_1)$ for threefolds with multiple nodal points. In detail, the derived category $\D(Y_1)$ of a resolution $Y_1$ of $X_1$ has a full exceptional collection and it descends to a SOD of $\D(X_1)$. The non-commutative algebra $R_1$ appearing in the SOD of $\D(X_1)$ is the indication of the node. For $X_2, X_3$ there are partial resolutions $Y_2, Y_3$ of a chosen nodal point, respectively. We can prove that the derived category of a chain of $m$ $\P^1$'s can be embedded into $\D(Y_m)$. This is how algebras $R_1$ and $R_{m-1}$ appear in the SOD of $\D(X_m)$ after descent, where $R_1$ is obtained in the same way as the $1$-nodal case $X_1$ and $R_{m-1}$ is the indication of the remaining nodal points.

It should be pointed out that Kawamata type SODs of $\D(X_m)$ constructed in this paper are certainly not unique. For example, there is another SOD for $\D(X_2)$ where the orthogonal complement of $\langle \O_{X_2}, \O_{X_2}(1)\rangle$ is given by $\langle \D(\kk), \D(R_2)\rangle$; see \cite[Remark 3.13]{psksoddp}. One would probably notice the asymmetry of the subcategory $\langle \D(R_1), \D(R_2)\rangle$ in $\D(X_3)$. The nodal point corresponding to $R_1$ is singled out because it is the chosen point that gets resolved. It is still unknown whether it is possible for $\D(X_3)$ to have $\langle \D(R_1), \D(R_1), \D(R_1) \rangle$ as the orthogonal complement of $\langle \O_{X_3}\rangle$ (instead of $\langle \O_{X_3}, \O_{X_3}(1)\rangle$). On the other hand, if one allows differential graded algebras in the SOD, then the derived category $\D(R_n)$ with $R_n, n\geqslant 1$ defined by (\ref{rn}) has a SOD with a copy of $\D(\kk)$ and $n$ copies of $\D(\kk[\epsilon]/\epsilon^2)$, where $\kk[\epsilon]/\epsilon^2$ is the \textit{derived} dual number with $\deg(\epsilon)=-1$. As a consequence, $\D(X_m)$ has a symmetric SOD with $4$ copies of $\D(\kk)$ and $m$ copies of $\D(\kk[\epsilon]/\epsilon^2)$. Since $X_m$ can be realized as codimension $3$ linear sections of $\Gr(2,5)$ (detailed in \S \ref{geometry}), these symmetric SODs also follow from the Homological Projective Duality theory, where the case for the smooth quintic del Pezzo threefold $X_0$ is done in \cite[\S 6.1]{kuzline}.

The key step of our construction of the Kawamata type SOD, which is interesting by itself, is the study of a quadric surface fibration with fibers of corank $2$. This method has been formalized and generalized in the subsequent paper \cite{xie-qsb}. More explicitly, the partial resolution $Y_3$ of $X_3$ at a nodal point is a quadric surface fibration $p_3\colon Y_3\to \P^1$ over $\P^1$. Because $p_3$ has a fiber of corank $2$, the non-trivial component of $\D(Y_3)$ is not well understood by known results. We show in Proposition \ref{3rdcase} that the non-trivial component is equivalent to the derived category of a chain of three $\P^1$'s. In this paper the proof is obtained by explicit computation. But the essential reason for such an equivalence is that $p_3\colon Y_3\to \P^1$ has a smooth section (each point of the section is a smooth point on the fiber) and a more theoretical proof is given in \cite[Example 6.1]{xie-qsb}.
\subsection{Organization of the paper}
In \S \ref{presec} we provide background materials used in the paper. In \S \ref{cliffordspinor} we review Clifford algebras and spinor sheaves of a quadratic form. In \S \ref{noncomproj} we introduce derived push-forward and pull-back functors for derived categories of noncommutative projective schemes and prove the projection formula (Lemma \ref{projf}). In \S \ref{derivedfacts} we present results on the derived push-forward functor of a proper morphism with fibers of dimension at most $1$ and explain in Proposition \ref{indsod} how to descend SODs via the derived push-forward functor. We also provide a generalization of Kawamata's result in Proposition \ref{kawaprop} that produces a non-trivial semiorthogonal component in the derived category of a projective threefold with multiple nodes.

In \S \ref{geometry} we study the geometry of $X_m$ and realize that there exists a small resolution $Y_m$ of $X_m$ at one of the nodal points and $Y_m$ is a quadric surface fibration over $\P^1$; see Proposition \ref{setup}. Concretely, let $x$ be a nodal point of $X_m$ and let $T_x X_m \cong \P^4\subset \P^6$ be the embedded projective tangent space of $X_m$ at $x$. Then the linear projection $\phi_m\colon X_m\dashrightarrow \P^1$ from $T_x X_m$ factors as $f_m^{-1}\circ p_m$ where $f_m\colon Y_m\to X_m$ is a small resolution of $X_m$ at $x$ and $p_m\colon Y_m\to \P^1$ is a quadric surface fibration.

In \S \ref{derivedcat} we study $\D(X_m)$ using the geometric model constructed in \S\ref{geometry}. The SOD of derived categories of quadric fibrations has been constructed by Kuznetsov \cite{kuzqfib}; see (\ref{ymsod}) for the SOD. The key step to understand $\D(Y_m)$ is to understand the nontrivial component $\D(\P^1, \B_{m,0})$, where $\B_{m,0}$ is the even part of the Clifford algebra of $p_m\colon Y_m \to \P^1$. They are well understood when $m=1,2$ because in these cases, fibers of $p_m$ have corank at most $1$. A new argument is used for $m=3$ where we show in Proposition \ref{3rdcase} that  $\D(\P^1, \B_{3,0})$ is equivalent to the derived category of a chain of three $\P^1$'s. From here we use Proposition \ref{indsod} and \ref{kawaprop} to descend the SOD of $\D(Y_m)$ to a SOD of $\D(X_m)$ along $f_{m*}\colon \D(Y_m) \to \D(X_m)$. One additional important ingredient is that for the exceptional locus $E\cong \P^1$ of $f_m\colon Y_m\to X_m$, its structure sheaf $\O_E$ has a Koszul resolution given by a regular section of the spinor bundle $\cS_E$ associated with $E$ (see Proposition \ref{spinorsheaves} (ii) and note that $E$ is a smooth section of $p_m$).

\subsection{Related work}
In preparation of the paper we learned that Pavic-Shinder \cite{psksoddp} are working on the same subject via a different approach. They study $\D(X_m)$ using a different rational map coming from $X_m$. In detail, they choose a line $L\subset X_m$ in the smooth locus and consider the linear projection $X_m\dashrightarrow \P^4$ from $L$. The image of the map is a smooth or nodal $3$-dimensional quadric $Q^3$ and the map $X_m\dashrightarrow Q^3$ factors as the inverse of the blow-up $Y\cong \Bl_L(X_m)\to X_m$ followed by the blow-up $Y=\Bl_C (Q^3)\to Q^3$ along a nodal curve $C$ of arithmetic genus $0$. In comparison, the rational map $\phi_m\colon X_m\dashrightarrow \P^1$ we used in this paper is a linear projection onto a line in the smooth locus of $X_m$. 

\subsection{Notations}\label{sec-notation}
We will use the following notations and conventions throughout the paper.

Given an algebraic scheme $X$, we denote by $\D(X)$ the bounded derived category of coherent sheaves on $X$. Denote by $\DD^-(X)$, $\DD(X)$ and $\DD^{\perf}(X)$ the bounded above, unbounded derived categories of coherent sheaves and derived categories of perfect complexes (quasi-isomorphic to a bounded complex of locally free sheaves of finite rank), respectively. Given a morphism $f\colon X\to Y$, we denote by $f_*$ and $f^*$ the total derived push-forward and pull-back functors, respectively. The underived push-forward and pull-back functors will be denoted by $R^0f_*$ and $L_0f^*$, respectively. 

Given a noncommutative projective scheme $(X, \cA_X)$ introduced in Definition \ref{noncommdef}, denote by $\Coh(X, \cA_X)$ and $\QCoh(X, \cA_X)$ the abelian categories of, respectively, coherent and quasi-coherent sheaves with right $\cA_X$-module structures. Note that when $\cA_X=\O_X$,  the pair $(X, \cA_X)$ is the usual scheme.

The base field $\kk$ is an algebraically closed field of characteristic $0$.

Let $\Gamma$ be a chain of $n$ $\P^1$'s. Let $\Gamma_i$ be the $i$-th component of $\Gamma$ and $d_i\in\mathbb{Z}, 1\leqslant i\leqslant n$.  Denote by
\begin{itemize}
    \item $\O_{\Gamma}\{d_1,\dots, d_n\}$ the line bundle on $\Gamma$ whose restriction to $\Gamma_i$ is $\O_{\Gamma_i}(d_i)$.
\end{itemize}

We will provide a geometric model for nodal quintic del Pezzo threefolds $X_m$ in Proposition \ref{setup} and it will be used across the paper. For the convenience of the reader we include a summary of related notations below.
\begin{itemize}
    \item $X_m$ is the quintic del Pezzo threefold with $m$ nodes for $m=1, 2, 3$;
    \item $f_m\colon Y_m\to X_m$ is a small resolution of a nodal point where $f_m$ contracts a smooth rational curve $E$ to the nodal point;
    \item $p_m\colon Y_m\to \P^1$ is a quadric surface fibration, where $E$ is a smooth section (consists of smooth points of the fibers) and the normal bundle $N_{E/Y_m}\cong \O_E(-1)^2$;
    \item Let $\E=\O_{\P^1}\oplus \O_{\P^1}(-1)^3$ and $\cL=\O_{\P^1}(-1)$. Let $\pi\colon \P(\E)\to \P^1$ be the projection. Then the total space $Y_m$ of the quadric surface fibration $p_m$ is the zero locus of a global section
    \begin{equation*}
        \sigma_m \in \Gamma(\P(\E), \O_{\P(\E)/\P^1}(2)\otimes \pi^*\cL) \cong \Gamma(\P^1, \Sym^2(\E^\vee) \otimes \cL)
    \end{equation*}
    on $\P(\E)$ where $\E^\vee$ is the dual of $\E$ and $\Sym^2$ is the second symmetric product. The exceptional locus $E$ of $f_m\colon Y_m\to X_m$ is the projectivization of $\O_{\P^1}\subset \E$.
\end{itemize}

Denote by
\begin{itemize}
    \item $i_m\colon Y_m\hookrightarrow \P(\E)$ the embedding ($p_m=\pi\circ i_m$);
    \item $\delta \in \Gamma(\P(\E), \O_{\P(\E)/\P^1}(1)\otimes \pi^*\E) \cong \Gamma(\P^1, \E^\vee\otimes \E)$ the section corresponding to the identity of $\End(\E)$;
    \item $\B_{m,0}$ the sheaf of even part of the Clifford algebra of the quadric fibration $p_m\colon Y_m\to\P^1$;
    \item $\B_{m,1}$ the sheaf of odd part of the Clifford algebra of $p_m$;
    \item $\Z_m$ the central subalgebra $\O_{\P^1}\oplus \Lambda^4\E \otimes (\cL^2)^\vee$ of $\B_{m,0}$;
    \item $g_m\colon C_m=\Spec_{\P^1}(\Z_m) \to \P^1$ the double cover ramified along the degeneration locus of $p_m$;
    \item $\wt{\B_{m, 0}}$ the unique sheaf of algebra over $C_m$ such that $g_{m*}(\wt{\B_{m,0}}) \cong \B_{m,0}$ as sheaves of $\Z_m$-algebras.
\end{itemize}

Note that the algebra structures of coherent sheaves $\B_{m,0}$, $\B_{m,1}$, $\wt{\B_{m,0}}$, $\Z_m$, the degeneration locus of $p_m\colon Y_m\to \P^1$ and therefore the map $g_m\colon C_m\to \P^1$ depend on the global section $\sigma_m$ defining  the quadric surface fibration $p_m$. 

\subsection{Acknowledgements}
I would like to thank Arend Bayer, Evgeny Shinder, Nebojsa Pavic for many helpful conversations. I am also grateful to Pavic and Shinder for communicating their ongoing work with me. I would also like to thank Alexander Kuznetsov for pointing out that the linear subspace $\P^4$ that $\phi_m\colon X_m\dashrightarrow \P^1$ is projecting from is the embedded projective tangent space of a chosen nodal point. Lastly, I would like to thank Arend Bayer and referees for the careful reading of the earlier draft and for suggestions of revision. The author is supported by the ERC Consolidator grant WallCrossAG, no. 819864.
\section{Preliminaries}\label{presec}
\subsection{Clifford algebras and spinor sheaves}\label{cliffordspinor}
In this section we review the Clifford algebra of a quadric hypersurface and spinor sheaves associated with linear subspaces of a quadric hypersurface, with a focus on quadric surfaces. Spinor sheaves on singular quadrics are studied by Addington in \cite{addspinor}. He follows the convention that spinor bundles on smooth quadrics are generated by global sections; e.g., the spinor bundle on $\P^1$ is $\O(1)$. Similarly, spinor sheaves of a quadric fibration can be constructed and they are used to study derived categories of nodal quintic del Pezzo threefolds in \S \ref{derivedcat}.

Let $V$ be a $\kk$-vector space and let $q$ be a quadratic form on $V$. The Clifford algebra of $q$ is
\begin{equation*}
    B_q=T^{\bullet}(V)/ \langle v\otimes v-q(v)\cdot 1\rangle_{v\in V}
\end{equation*}
where $T^\bullet (V)$ is the free associative algebra generated by $V$. We view $T^\bullet (V)$ as a graded algebra with elements of $V$ having degree $1$. Then $B_q$ has a natural $\mathbb{Z}/2$-grading: the even part $B_{q0}$ is spanned by monomials of even degrees and the odd part $B_{q1}$ is spanned by monomials of odd degrees.

Let $W$ be an isotropic subspace of $V$, that is, $q|_W=0$. Then the subalgebra of $B_q$ generated by $W$ is $\Lambda^\bullet W$. Let $I^W$ be the right ideal $(\Lambda^{\dim W} W) \cdot B_q$ of $B_q$ generated by $\Lambda^{\dim W} W$. Let $I^W=I^W_0\oplus I^W_1$ be the decomposition into the even part $I^W_0$ and the odd part $I^W_1$. Consider the map of vector bundles
\begin{equation*}
    \begin{array}{rcl}
        \O_{\P(V)}(-1)\otimes I^{W} & \xrightarrow{{\delta^W}} & \O_{\P(V)}\otimes I^{W}\\
        v\otimes \xi & \mapsto & 1\otimes \xi v\\
    \end{array}
\end{equation*}
where $\O_{\P(V)}(-1)$ is regarded as the universal subbundle of $\O_{\P(V)}\otimes V$. We use $\delta^W(n)$ to denote the map obtained by tensoring $\delta^W$ with $\O_{\P(V)}(n)$. Since $\delta^W(n+1)\circ \delta^W(n)=q$ for every $n$, we have $\delta^W$ is an isomorphism away from the quadric $Q=\{q=0\} \subset \P(V)$. Because $\ker{(\delta^W)}$ is torsion and $\O_{\P(V)}(-1)\otimes I_{W}$ is torsion free, we have $\delta^W$ is injective. Hence, there are short exact sequences
\begin{equation}\label{spinor}
    \begin{split}
        & 0\to \O_{\P(V)}(-1)\otimes I^W_0 \xrightarrow{\delta^W_0} \O_{\P(V)}\otimes I^W_1 \to S_W\to 0,\\
        & 0\to \O_{\P(V)}(-1)\otimes I^W_1 \xrightarrow{\delta^W_1} \O_{\P(V)}\otimes I^W_0 \to T_W \to 0.\\
    \end{split}
\end{equation}
where $\delta^W=\delta^W_0 \oplus \delta^W_1$ and $S_W\coloneqq\coker{(\delta^W_0)}, T_W\coloneqq\coker{(\delta^W_1)}$ are supported on $Q$. The map $\delta^W_1$ can be induced from $\delta^W_0$ because
\begin{equation*}
    \delta^W_1 \cong \delta^W_0\otimes_{B_{q0}} B_{q1}.
\end{equation*}
In particular, $I^W_1 \cong I^W_0\otimes_{B_{q0}} B_{q1}$. There are resolutions on $Q$ for $S_W$ (and similarly for $T_W$):
\begin{equation}\label{spinorres}
    \begin{split}
        \cdots \xrightarrow{\delta^W_0(-2)} \O_Q(-2) \otimes I^W_1 \xrightarrow{\delta^W_1(-1)} \O_Q(-1) \otimes I^W_0 \xrightarrow{\delta^W_0} \O_Q\otimes I^W_1 \to S_{W} \to 0,\\
        0\to S_{W}\to \O_Q(1)\otimes I^W_0 \xrightarrow{\delta^W_0(2)} \O_Q(2) \otimes I^W_1 \xrightarrow{\delta^W_1(3)} \O_Q(3)\otimes I^W_0 \xrightarrow{\delta^W_0(4)} \cdots.\\
    \end{split}
\end{equation}
The sheaves $S_{W}, T_{W}$ constructed here are called the \textit{spinor sheaves} of $Q$ associated with the linear subspace $\P(W)\subset Q$. They are the generalization of spinor bundles on smooth quadrics. Similarly for a flat quadric fibration $\cQ \subset \P(\cV) \to S$ and an isotropic subbundle $\mathcal{W}\subset \cV$, one can construct spinor sheaves $\cS_{\mathcal{W}}, \mathcal{T}_{\mathcal{W}}$ of $\cQ$ associated with $\P(\mathcal{W})\subset \cQ$. For details of this construction, see \cite[\S2.2]{xie-qsb}.

\begin{prop}[{\cite[Proposition 2.1 and 4.1]{addspinor}}]\label{spinorprop}
    The spinor sheaves $S_{W}, T_{W}$ constructed by (\ref{spinor}) are reflexive sheaves on $Q$. Let $K\subset V$ be the kernel of $q$; i.e., the singular locus of $Q$ is $\P(K)$. We have

    (i) if $\P(K)\cap \P(W)=\emptyset$ and $\codim(W)>1$, then $S_{W}, T_{W}$ are locally free sheaves of rank $2^{\codim(W)-2}$ on $Q$;

    (ii) if $\codim(W)$ is odd, then $S_{W}^\vee\cong S_{W}(-1)$ and $T_W^\vee\cong T_W(-1)$;

    (iii) if $\codim(W)$ is even, then $S_{W}^\vee\cong T_{W}(-1)$.
\end{prop}

The sheaves $S_W, T_W$ are unchanged when we vary $W$ continuously with $W\cap K$ fixed. They are not isomorphic when the family of $W$ with fixed $\dim W$ and $W\cap K$ is not connected. Let $\pi\colon V\to V/K$ be the projection map. Then we have $S_W\ncong T_W$ when $\dim \pi(W)= \frac{1}{2}\dim V/K$ because there are two connected families of $W$, and $S_W\cong T_W$ when $\dim \pi(W)< \frac{1}{2} \dim V/K$.

For the rest of the section we will focus on $\dim V=4$.

\begin{prop}\label{spinorsheaves}
    (i) Let $Q=\{q=0\}\subset \P(V)\cong \P^3$ be a reduced quadric surface; i.e., it has corank at most $2$. For a smooth point $x\in Q$, let $S_x$ be the spinor sheaf for the corresponding $1$-dimension isotropic subspace constructed by (\ref{spinor}). Then $S_x$ is a rank $2$ vector bundle on $Q$ such that $S_x^\vee \cong S_x(-1)$ and the skyscraper sheaf $\O_x$ has a resolution:
    \begin{equation}\label{koszulres}
        0\to \det (S_x(-1)) \to S_x(-1) \to \O_Q \to \O_x\to 0.
    \end{equation}

    (ii) Let $p\colon \cQ\to S$ be a flat quadric surface fibration with fibers of corank at most $2$ and let $q\colon \cV\to \sL$ be the corresponding quadratic form. Assume that the base scheme $S$ is Cohen-Macaulay and there exists an isotropic sub line bundle $\cN\subset \cV$ such that $F=\P(\cN)\subset \cQ$ is a \textit{smooth} section; i.e., $F$ consists of smooth points on the fibers. Then the spinor sheaf $\cS_F$ associated with $F$ is a rank $2$ vector bundle on $\cQ$ such that
    \begin{equation*}
        \cS_F^\vee \cong \cS_F \otimes \O_{\cQ/S}(-1)\otimes p^*(\cN^\vee\otimes \det(\cV^\vee)\otimes \sL)
    \end{equation*}
     and there is a Koszul resolution
    \begin{equation*}
        0\to \det(\cS_F^\vee) \to \cS_F^\vee \to \O_{\cQ}\to \O_F\to 0.
    \end{equation*}
\end{prop}
\begin{proof}
    (i) Let $W=\kk v$ be the $1$-dimensional isotropic subspace of $V$ that corresponds to $x$. Then $S_x=S_{W}$. By Proposition \ref{spinorprop} $S_x$ is a rank $2$ vector bundle on $Q$ and $S_x^\vee\cong S_x(-1)$. Explicitly, we have the following three cases.
    \begin{enumerate}
        \item $Q$ is smooth: $Q\cong\P^1\times\P^1$ and $S_x\cong \O_{\P^1\times\P^1}(1,0) \oplus \O_{\P^1\times\P^1}(0,1)$.
        \item $Q$ is of corank $1$: $Q$ is the cone over a smooth conic $C\cong\P^1$. Let $y\in Q$ be the vertex. Let $\pi_y\colon Q-\{y\} \to C$ be the projection and let $j_y\colon Q-\{y\} \hookrightarrow Q$ be the open embedding. Denote the rank $1$ spinor sheaf $R^0j_{y*}\pi_y^* \O_{\P^1}(1)$ on $Q$ by $M$. Then $S_x$ is the unique nontrivial extension of $M$ by $M$; cf. \cite[Example 5.5]{kawamulti}.
        \item $Q$ is of corank $2$: $Q$ is the union of two $\P^2$'s intersecting along $\P^1$. Since $x$ is a smooth point, $x$ is not a point on the intersection $\P^1$. Denote the $\P^2$ containing $x$ by $\P^2_1$ and the $\P^2$ not containing $x$ by $\P^2_2$. Denote by $i_l\colon \P^2_l\hookrightarrow Q$ the inclusions for $l=1,2$. Then using the sequences (\ref{spinor}) defining spinor sheaves, one gets $L_0i_1^*S_x\cong T_{\P^2_1}(-1)$ where $T_{\P^2_1}$ is the tangent bundle and $L_0i_2^*S_x \cong \O_{\P^2_2}\oplus \O_{\P^2_2}(1)$.
    \end{enumerate}
    Let $\{v_i\}_{i=1}^4$ be an orthogonal basis of $V$ for $q$ such that $v_1=v$. Let $s_v\in H^0(Q, S_x)$ be the section corresponding to the map $\O_Q v \subset \O_Q \otimes I^W_1\to S_{W}=S_x$ from (\ref{spinorres}). Note that the composition $\O_Qv\to S_x \to \O_Q(1)\otimes I^W_0$ is given by the column vector $(v_2, v_3, v_4, 0)$. Thus, the zero locus of $s_v$ is $\{v_2=v_3=v_4=0\}=\{x\}$. Because both the codimension of the point $x$ and the rank of $S_x$ are $2$, this implies that $s_v$ is a regular section and the skyscraper sheaf $\O_x$ has the Kozsul resolution (\ref{koszulres}).

    (ii) Let $\B_{q0}$ and $\B_{q1}$ be the even and odd Clifford algebras of $q$, respectively. Let $\cI_0\subset \B_{q0}$, $\cI_1\subset \B_{q1}$ be right modules over $\B_{q0}$ generated by $\cN$. We have decompositions
    \begin{equation*}
        \B_{q0} \cong \O_S \oplus \Lambda^2 \cV \otimes \sL^\vee \oplus \det \cV \otimes (\sL^2)^\vee, \quad \B_{q1} \cong \cV \oplus \Lambda^3 \cV \otimes \sL^\vee.
    \end{equation*}
    Taking the parts that have the factor $\cN$ gives the decompositions
    \begin{equation*}
        \cI_0 \cong \cN\otimes \cV/\cN\otimes \sL^\vee \oplus \det \cV\otimes (\sL^2)^\vee, \quad \cI_1 \cong \cN \oplus \cN\otimes\Lambda^2 (\cV/\cN)\otimes \sL^\vee.
    \end{equation*}
    Moreover, one has
    \begin{equation}\label{dual} 
        \cI_n^\vee \cong \cI_{1-n}^\circ \otimes \cN^\vee\otimes \det \cV^\vee\otimes \sL^2
    \end{equation} 
    for $n=0,1$ where $\cI_n^\circ$ are the left modules over $\B_{q0}$ generated by $\cN$ and $\cI_n^\circ\cong \cI_n$ as $\O_S$-modules. For $n\in\mathbb{Z}$ let 
    \begin{equation*}
        \cI_n\coloneqq\left\{
        \begin{array}{ll}
            \cI_0\otimes \sL^k, & n=2k,\\
            \cI_1\otimes \sL^k, & n=2k+1.
        \end{array}
         \right.
    \end{equation*}
    The relative version of (\ref{spinorres}) reads as
    \begin{equation*}
        \begin{split}
        & \dots \to \O_{\cQ/S}(-1)\otimes p^*\cI_0\to p^* \cI_1\to \cS_F\to 0,\\
        & 0\to \cS_F\to \O_{\cQ/S}(1)\otimes p^*\cI_2\to \O_{\cQ/S}(2)\otimes p^*\cI_3\to \cdots .
        \end{split}
    \end{equation*}
    Taking the dual of the second sequence, one gets
    \begin{equation*}
        \dots \to \O_{\cQ/S}(-2)\otimes p^*\cI_3^\vee\to \O_{\cQ/S}(-1)\otimes p^*\cI_2^\vee\to \cS_F^\vee\to 0.
    \end{equation*}
    Since $\cI_3^\vee \cong \cI_1^\vee\otimes \sL^\vee$ and $\cI_2^\vee \cong \cI_0^\vee\otimes \sL^\vee$, combining with (\ref{dual}) we have $\cS_F^\vee \cong \cS_F \otimes \O_{\cQ/S}(-1)\otimes p^*(\cN^\vee\otimes \det(\cV^\vee)\otimes \sL)$.

    The proof for the Koszul resolution is the same to that of (i). Note that the choice of the regular section $s_v\in H^0(Q, S_x)$ is canonical and thus the argument can be extended to the fibration case.
\end{proof}

Lastly, we include the following computation of even Clifford algebras for later reference.

\begin{ex}\label{qfex}
    Let $q(x)=x_1x_2+q'(x_3,x_4)$ be a quadratic form of $V$ with $\dim V=4$. Then $q$ is the direct sum of the hyperbolic quadratic form $u=x_1x_2$ and $q'$. We get
    \begin{equation*}
        B_q\cong B_u\otimes B_{q'} \cong \M_2(B_{q'}), \quad B_{q0} \cong B_{u0}\otimes B_{q'0}\oplus B_{u1}\otimes B_{q'1}
    \end{equation*}
    where $M_2$ is the matrix algebra of size $2$.

    Up to a change of variables $q'(x_3, x_4)$ is of the form $\lambda x_3^2+ \mu x_4^2$ for some $\lambda, \mu\in \kk$. The symmetric matrix corresponding to $q$ is
    \begin{equation*}
        \left(
        \begin{array}{cc}
            \left(\begin{array}{cc} 0 &  1\\ 1 & 0\end{array}\right) &  0\\
            0 & \left(\begin{array}{cc} 2\lambda &  0\\ 0 & 2\mu\end{array}\right)
        \end{array}
        \right).
    \end{equation*}
    Let $\{v_1, v_2, e_3, e_4\}$ be a dual basis of $\{x_i\}_{i=1}^4$. Then we have
    \begin{equation*}
        v_1^2=v_2^2=0, \quad v_1v_2+v_2v_1=1, \quad v_ie_j=-e_jv_i, \quad i\in \{1,2\},\;  j\in \{3,4\} \quad \text{in} \; B_q.
    \end{equation*}
    If we set $e_1=v_1+v_2$ and $e_2=v_2-v_1$, then $\{e_i\}_{i=1}^4$ is an orthogonal basis of $V$ for $q$; i.e.,
    \begin{equation*}
        e_i e_j=-e_j e_i, \quad i\neq j, \quad e_i^2=q(e_i) \quad \text{in} \; B_q.
    \end{equation*}
    Moreover, we have $q(e_1)=1$ and $q(e_2)=-1$.  

    Let $d=e_1e_2e_3e_4\in B_q$. Then $Z_q \coloneqq\kk 1 \oplus \kk d \cong \kk[d]/(d^2-\det(q))$ is a central subalgebra of $B_{q0}$. If $q'(x_3,x_4)=0$, then 
    \begin{equation}\label{corank2}
        \begin{split}
            B_{q0} & \cong Z_qv_1v_2\oplus Z_qv_2v_1 \oplus Z_q v_1e_3 \oplus Z_q v_1e_4 \oplus Z_q v_2e_3 \oplus Z_q v_2e_4\\
            & \cong Z_qv_1v_2\oplus Z_qv_2v_1 \oplus \kk v_1e_3 \oplus \kk v_1e_4 \oplus \kk v_2e_3 \oplus \kk v_2e_4.
        \end{split}
    \end{equation}
    If $q'(x_3, x_4)\neq 0$, then there exists $v_3\in \kk e_3\oplus \kk e_4$ such that $q(v_3)=1$. In this case, 
    \begin{equation*}
        B_{q0} \cong \frac{Z_q \langle s,t \rangle}{\langle s^2=1, \, t^2=1, \, st=-ts \rangle}
    \end{equation*} 
    where  $s=e_1e_2, \, t=e_2v_3$.
    Morever, $B_{q0} \cong \M_2(Z_q)$ and an explicit identification is given by
    \begin{equation*}
        s=
        \left(
        \begin{array}{lr}
            1 & 0\\
            0 & -1 \\
        \end{array}
        \right)
        ,\quad
        t=
        \left(
        \begin{array}{lr}
            0 & 1\\
            1 & 0 \\
        \end{array}
        \right).
    \end{equation*}
    Under this identification we have
    \begin{equation}\label{matrix}
        \begin{split}
            & \left(
            \begin{array}{lr}
                1 & 0\\
                0 & 0 \\
            \end{array}
            \right)
            =\frac{s+1}{2}=v_1v_2,
            \quad
            \left(
            \begin{array}{lr}
            0 & 0\\
            0 & 1 \\
            \end{array}
            \right)
            =1-v_1v_2=v_2v_1,\\
            & \left(
            \begin{array}{lr}
                0 & 1\\
                0 & 0 \\
            \end{array}
            \right)
            =\frac{t+st}{2}=-v_1v_3,  
            \quad
            \left(
            \begin{array}{lr}
                0 & 0\\
                1 & 0 \\
            \end{array}
            \right)
            =\frac{t-st}{2}=v_2v_3.\\
        \end{split}
    \end{equation}
    Furthermore,
    \begin{equation}\label{actionbycenter}
        \begin{split}
            & Z_q v_1v_2=\kk v_1v_2\oplus \kk v_1v_2e_3e_4, \quad Z_q v_2v_1=\kk v_2v_1\oplus \kk v_2v_1e_3e_4,\\
            & Z_q v_1v_3=\kk v_1e_3\oplus \kk v_1e_4, \quad Z_q v_2v_3=\kk v_2e_3\oplus \kk v_2e_4.
        \end{split}
    \end{equation}
\end{ex}
\subsection{Noncommutative projective schemes}\label{noncomproj}
In this section we introduce noncommutative projective schemes in the sense of pairs of projective schemes together with sheaves of algebras. We will define the derived push-forward and pull-back functors of a morphism and prove the projection formula (Lemma \ref{projf}) in this setting. The generalization to noncommutative schemes in the sense of pairs of noetherian schemes together with sheaves of algebras is given in Appendix A of \cite{xie-qsb}. 

\begin{defn}\label{noncommdef}
    A pair $(X, \cA_X)$ is a \textit{noncommutative projective scheme} over $\kk$ if $X$ is a projective scheme over $\kk$ and $\cA_X$ is a quasi-coherent $\O_X$-module and sheaf of $\O_X$-algebras. A \textit{morphism} $\Theta=(\theta, \theta_{\cA})\colon (X, \cA_X)\to (Y, \cA_Y)$ of noncommutative projective schemes over $\kk$ consists of a morphism $\theta\colon X\to Y$ of schemes and a homomorphism $\theta_{\cA}\colon L_0\theta^*\cA_Y\to \cA_X$ of $\O_X$-algebras. 
\end{defn}
Following the same lines after Definition 10.3 in \cite{kuzhs}, there is the derived functor
\begin{equation*}
    \Theta_*=\theta_*\colon \DD(\QCoh(X,\cA_X))\to \DD(\QCoh(Y, \cA_Y)).
\end{equation*}

For every $G\in \QCoh(Y, \cA_Y)$, 
\begin{equation*}
    L_0\Theta^*G\coloneqq \theta^{-1} G \otimes_{\theta^{-1} \cA_Y} \cA_X \cong L_0 \theta^*G\otimes_{L_0\theta^* \cA_Y} \cA_X \in  \QCoh(X, \cA_X).
\end{equation*}
Since $Y$ is projective, there is a locally free $\O_Y$-module $\cV$ (of finite rank if $G$ is a coherent $\O_Y$-module) and an $\O_Y$-module epimorphism $\cV \overset{s}{\twoheadrightarrow} G$. Thus, there is a right $\cA_Y$-module epimorphism
\begin{equation*}
    \cV\otimes \cA_Y \twoheadrightarrow G, \quad v\otimes a\mapsto s(v)a, \quad v\in\cV, a\in \cA_Y
\end{equation*}
from the locally free $\cA_Y$-module $\cV\otimes \cA_Y$. This means that $\QCoh(Y, \cA_Y)$ has enough locally free objects. By Lemma 13.29.1 and the proof of Lemma 20.26.12 in \cite{stacks-project}, for every complex $G\in \DD(\QCoh(Y,\cA_Y))$, there is a locally free resolution $K\in \DD(\QCoh(Y,\cA_Y))$. We can define the derived functor $\Theta^*$ of $L_0\Theta^*$ by
\begin{equation*}
    \Theta^* G\coloneqq L_0\Theta^*K.
\end{equation*}
This gives a well-defined functor
\begin{equation*}
    \Theta^*\colon \DD(\QCoh(Y, \cA_Y)) \to \DD(\QCoh(X, \cA_X)).
\end{equation*}
It is clear that $\Theta^*$ is a left adjoint of $\Theta_*$.

\begin{lemma}[Projection formula]\label{projf}
    Let $\Theta=(\theta, \theta_{\cA})\colon (X, \cA_X)\to (Y, \cA_Y)$ be a morphism between noncommutative projective schemes defined by Definition \ref{noncommdef}. For every $F\in \DD(\QCoh(X, \cA_X^{\op}))$ and $G\in \DD(\QCoh(Y, \cA_Y))$ the natural map
    \begin{equation}\label{projformula}
        G \otimes_{\cA_Y}^\LL \Theta_*(F) \to \Theta_*(\Theta^*(G) \otimes_{\cA_X}^\LL  F)
    \end{equation}
    is an isomorphism in $\DD(\QCoh(Y))$.
\end{lemma}
\begin{proof}
    There exists a natural map
    \begin{equation*}
        \Theta^*(G \otimes_{\cA_Y}^\LL \Theta_*(F))\cong \Theta^*(G) \otimes_{\cA_X}^\LL \Theta^*\Theta_*(F)  \to \Theta^*(G) \otimes_{\cA_X}^\LL  F
    \end{equation*}
    because $\Theta^*$ commutes with derived tensor functors and $(\Theta^*, \Theta_*)$ are adjoint. Applying the adjointness again gives (\ref{projformula}). Let $T\subset \DD(\QCoh(Y, \cA_Y))$ be the full triangulated subcategory of objects $G$ such that (\ref{projformula}) is an isomorphism. We will prove that $T= \DD(\QCoh(Y, \cA_Y))$. 

    Note that $\DD(\QCoh(Y))$ has a compact generator $U\in \DD^{\perf}(Y)$ by \cite[Theorem 3.1.1]{bonvdbnoncom} and is cocomplete (admits arbitrary direct sums). Then $\DD(\QCoh(Y, \cA_Y))$ has a compact generator $U\otimes \cA_Y$ and is also cocomplete. Clearly $U\otimes \cA_Y\in T$ and $T$ is closed under arbitrary direct sums because $\Theta^*, \Theta_*, \otimes_{\cA_Y}^\LL, \otimes_{\cA_X}^\LL$ all commute with direct sums. We deduce from \cite[Theorem 2.2c]{luncatres} that $T=\DD(\QCoh(Y, \cA_Y))$.
\end{proof}

We end the section with a few simple examples. Let $X$ be a projective scheme.
\begin{ex}
    Let $\cV$ be a vector bundle on $X$. Then $\cV$ is a locally projective left module over $\sEnd(\cV)$ and  there is the Morita equivalence $\DD(X)\cong \DD(X, \sEnd(\cV))$ given by inverse equivalences $\cV^\vee \otimes_{\O_X} -$ and $-\otimes_{\sEnd(\cV)} \cV$.
\end{ex}
\begin{ex}
    Let $\cC$ be a coherent sheaf of commutative algebras on $X$. Then $\DD(X,\cC)\cong \DD(\Spec(\cC))$.
\end{ex}
\begin{ex}
    Let $\B$ be a coherent sheaf of algebras on $X$ and let $\cC$ be a central subalgebra of $\B$; i.e., $\cC$ is contained in the center of $\B$. Let $f\colon \Spec(\cC)\to X$ be the natural map and let $\wt{\B}$ be the unique sheaf of algebras on $\Spec(\cC)$ such that $f_*\wt{\B}\cong \B$ as sheaves of $\cC$-algebras. Then $\DD(X, \B)\cong \DD(\Spec(\cC), \wt{\B})$.
\end{ex}
\subsection{Some facts about derived categories}\label{derivedfacts}
We first provide a useful lemma that provides an easy criterion for semiorthogonal decompositions of derived categories of Gorenstein projective schemes to be admissible and provides a mutation for such decompositions. 

Recall that given a map $f\colon X\to S$, a triangulated subcategory $\cA\subset \D(X)$ is called $S$-linear if $\cF\otimes f^*\cG \in \cA$ for every $\cF\in \cA$ and $\cG\in \DD^{\perf}(S)$. A semiorthogonal decomposition of $\D(X)$ is $S$-linear if each component is $S$-linear, and is admissible if each component is admissible.

\begin{lemma}\label{adsodlem}
    Let $f\colon X\to S$ be a map between Gorenstein projective schemes. Assume that there is an $S$-linear semiorthogonal decomposition
    \begin{equation}\label{gsod}
        \D(X)=\langle \cA_1, \dots, \cA_n \rangle
    \end{equation}
    such that all but one components are contained in $\DD^{\perf}(X)$; i.e., there exists some $t$, $1\leqslant t\leqslant n$ such that $\cA_j\subset \DD^{\perf}(X)$ for every $j\neq t$. Then the relative dualizing sheaf $\omega_{X/S}$ is a line bundle and the semiorthogonal decomposition (\ref{gsod}) is admissible. In addition, there is an admissible $S$-linear semiorthogonal decomposition
    \begin{equation}\label{sodmu}
        \D(X) =\langle \cA_n \otimes \omega_{X/S}, \cA_1, \dots, \cA_{n-1} \rangle.
    \end{equation}
\end{lemma}
\begin{proof}
    Since $X,S$ are Gorenstein projective, dualizing sheaves $\omega_X, \omega_S$ are line bundles. Then $\omega_{X/S}$ is a line bundle because $\omega_{X/S}\cong \omega_X\otimes f^*\omega_S^\vee$. Set $\cA=\langle \cA_1, \dots, \cA_{n-1} \rangle$. By assumption we have $\cA$ or $\cA_n$ is contained in $\DD^{\perf}(X)$. Then by \cite[Lemma 2.15]{kpsk-1} $\cA_n$ is admissible and there is a SOD
    \begin{equation*}
        \D(X)=\langle \cA_n\otimes \omega_X, \cA \rangle = \langle \cA_n\otimes \omega_X, \cA_1, \dots, \cA_{n-1}\rangle,
    \end{equation*}
    which by definition is $S$-linear. Since $\omega_X\cong \omega_{X/S}\otimes f^*\omega_S$ and $\cA_n$ is $S$-linear, we have $\cA_n\otimes \omega_X\cong \cA_n\otimes \omega_{X/S}$. This gives the SOD (\ref{sodmu}). Repeating this process, we get $\cA_i$ is admissible for each $i, 1\leqslant i\leqslant n$ and an $S$-linear SOD
    \begin{equation*}
        \D(X) =\langle \cA_i \otimes \omega_{X/S}, \dots, \cA_n \otimes \omega_{X/S}, \cA_1, \dots, \cA_{i-1} \rangle.
    \end{equation*}
    Thus, (\ref{gsod})(\ref{sodmu}) are admissible SODs.
\end{proof}

In the next two propositions we present the behavior of the derived push-forward functor of certain proper morphism with fibers of dimension at most $1$ and how it induces semiorthogonal decompositions.
\begin{prop}[{\cite[Theorem 7.13]{bbnoncomdef}}]\label{pushforward}
    Let $\gamma\colon Y\to X$ be a proper morphism of normal varieties over $\kk$ with fibers of dimension at most $1$ such that (the derived push-forward) $\gamma_* \O_Y \cong \O_X$. Let $X^1$ be the locus of $X$ where fibers of $\gamma$ are one-dimensional and assume that $X^1$ is a finite set of points. For each closed point $x\in X^1$, let $C_x\coloneqq\gamma^{-1}(x)_{\text{red}}$ be the reduced fiber over $x$ and let $C_{x,i}$ be its irreducible components. Denote by $l_{x,i}\colon C_{x,i}\hookrightarrow Y$ the embeddings. Then each $C_{x,i}$ is a smooth rational curve and $\gamma_*\colon \D(Y) \to \D(X)$ is a Verdier quotient with kernel $\ker(\gamma_*)=\langle l_{x, i*} \O_{C_{x, i}}(-1) \rangle_{x\in X^1}$; i.e., we have an equivalence $\D(Y)/\ker(\gamma_*)\cong \D(X)$.
\end{prop}
\begin{proof}
    The condition $\gamma_* \O_Y \cong \O_X$ implies that $H^1(C_x, \O_{C_x})=0$ for every $x\in X^1$. Note that the base field $\kk$ is assumed to be algebraically closed. Thus, $C_{x,i}$ is a smooth rational curve by \cite[Theorem D.1]{bbflop}.

    Now we prove the claims for the derived push-forward functor $\gamma_*$. Since the fibers of $\gamma\colon Y\to X$ have dimension at most $1$, the spectral sequence $R^i\gamma_*\cH^j(\cF)\Rightarrow R^{i+j}\gamma_*\cF$, where $\cH^j$ is the cohomology sheaf in degree $j$, degenerates and we have short exact sequences
    \begin{equation}\label{specseq}
        0\to R^1\gamma_* \cH^{j-1}(\cF) \to R^j\gamma_* \cF \to R^0\gamma_* \cH^j(\cF)\to 0.
    \end{equation}
    This implies that $\ker(\gamma_*)$ is generated by $\ker(\gamma_*)\cap \Coh(Y)$. Since the one-dimensional fibers of $\gamma$ are isolated by assumption, every $\cF\in \ker(\gamma_*)\cap \Coh(Y)$ is the direct sum of sheaves supported on $C_x, x\in X^1$. Together with Proposition 7.12 in \cite{bbnoncomdef} this implies that $\ker(\gamma_*)\cap \Coh(Y)$ has finite length. Therefore, Theorem 7.13 in \cite{bbnoncomdef} implies that $\ker(\gamma_*)=\langle l_{x, i*} \O_{C_{x, i}}(-1) \rangle_{x\in X^1}$. 

    For the equivalence of the induced functor $\D(Y)/\ker(\gamma_*)\to \D(X)$, we use an argument similar to \cite[Theorem 2.14]{bks-schober}. We will make use of the canonical truncation functors $\tau^{\leqslant p}, \tau^{\geqslant p}$. The sequence (\ref{specseq}) implies that if $\tau^{\leqslant p} \cG =0$ for $\cG\in \D(X)$, then
    \begin{equation}\label{truncate}
        \gamma_*(\tau^{\leqslant p-1}\gamma^*\cG)=0,\quad \text{i.e.,} \quad \cG\cong \gamma_*\gamma^*\cG \cong \gamma_*(\tau^{\geqslant p}\gamma^*\cG).
    \end{equation}
    We deduce from this that $\gamma_*$ is essentially surjective. Finally, we claim that for $\cF, \cF'\in\D(Y)$, 
    \begin{equation*}
        \Hom_{\D(Y)/\ker(\gamma_*)}(\cF, \cF')\to \Hom_{\D(X)}(\gamma_*\cF, \gamma_*\cF')
    \end{equation*}
    is a bijection. The inverse map is constructed as follows. Choose $p$ such that $\tau^{\leqslant p} \cF=0$ and $\tau^{\leqslant p-1}\cF'=0$. A map $f\colon \gamma_*\cF \to \gamma_*\cF'$ gives a map
    \begin{equation*}
        g\colon \tau^{\geqslant p}\gamma^*\gamma_*\cF \to \tau^{\geqslant p}\cF'\cong \cF'.
    \end{equation*}
    Let $K$ be the cone of $\gamma^*\gamma_* \cF\to \cF$, which is bounded above but not necessarily bounded. Taking $\cG=\gamma_*\cF$ in (\ref{truncate}) we get
    \begin{equation*}
        \gamma_*\cF \cong \gamma_*(\tau^{\geqslant p}\gamma^*\gamma_*\cF).
    \end{equation*}
    Since $\cF \cong \tau^{\geqslant p}\cF$, we get $\gamma_*\tau^{\geqslant p} K=0$ and thus $\tau^{\geqslant p} K\in\ker(\gamma_*)$. This implies that the map $g$ corresponds to a map $\cF\to \cF'$ in the Verdier quotient $\D(Y)/\ker(\gamma_*)$.
\end{proof}

It is explained in \cite[\S2]{kks-surf} how one can induce a SOD of the derived category of a surface with rational singularities from a SOD of the derived category of its resolution. We claim that it can be applied to a proper morphism of dimension at most $1$.
\begin{prop}\label{indsod}
    Let $\gamma \colon Y\to X$ be the map in Proposition \ref{pushforward}. Suppose there is a semiorthogonal decomposition
    \begin{equation}\label{sody}
        \D(Y)= \langle \wt{\cA}_1, \dots, \wt{\cA}_n \rangle
    \end{equation}
    and assume that for every irreducible component $C_{x,i}, x\in X^1$, there exists some $j, 1\leqslant j \leqslant n$ such that $\O_{C_{x, i}}(-1)\in \wt{\cA}_j$. 

    (i) There is a semiorthogonal decomposition
    \begin{equation}\label{sodx}
        \D(X) =\langle \cA_1, \dots, \cA_n \rangle
    \end{equation}
    where $\cA_i = \gamma_*(\wt{\cA}_i)\cong \wt{\cA}_i/(\wt{\cA}_i\cap \ker(\gamma_*))$. 

    (ii) Assume that $Y$ is Gorenstein projective and all but one components of (\ref{sody}) are contained in $\DD^{\perf}(Y)$. Let $K_Y$ be the canonical Cartier divisor. If $K_Y. C_{x_i,j}=0$ for every $i,j$, then the SOD (\ref{sodx}) is admissible. 
\end{prop}
\begin{proof}
    (i) By \cite[Proposition 4.1]{kuzshcatab}, it suffices to show that $\ker(\gamma_*)\subset \D(Y)$ is compatible with (\ref{sody}); i.e., there is a SOD
    \begin{equation*}
        \ker(\gamma_*) = \langle \ker(\gamma_*)\cap \wt\cA_1, \dots, \ker(\gamma_*)\cap \wt{\cA}_n \rangle.
    \end{equation*}
    Clearly the collection on the right hand side is semiorthogonal. It generates $\ker(\gamma_*)$ by Proposition \ref{pushforward} and the assumption. 

    (ii) Since $K_Y. C_{x_i,j}=0$ for every $i, j$, we have if $E\in D_i$, then $\O_E(-1)\in \wt{\cA}_i(\pm K_Y)$. Moreover, by Lemma \ref{adsodlem} (take $S$ as a point) we get SODs
    \begin{equation*}
        \begin{split}
            & \D(Y)= \langle \wt{\cA}_{i+1}(K_Y), \dots, \wt{\cA}_n(K_Y), \wt{\cA}_1, \dots, \wt{\cA}_i \rangle,\\
            & \D(Y)= \langle \wt{\cA}_i, \dots, \wt{\cA}_n, \wt{\cA}_1(-K_Y), \dots, \wt{\cA}_{i-1}(-K_Y) \rangle
        \end{split}
    \end{equation*}
    for every $i$. Theses SODs of $\D(Y)$ also induce SODs of $\D(X)$ via $\gamma_*$ by (i). Thus, each $\cA_i\cong \gamma_*(\wt{\cA}_i)$ is both left and right admissible, which concludes (ii).
\end{proof}

We also provide a generalization of the criterion given by Kawamata in \cite{kawaodp} for the descent of SODs. This will be used to descend the SOD of $\D(Y_m)$ in Corollary \ref{ymsodnew} to a SOD of $\D(X_m)$ in Theorem \ref{mainthm}.
\begin{prop}[{\cite[Theorem 6.1 and 5.1]{kawaodp}}] \label{kawaprop}
    Let $X$ be a Gorenstein projective $3$-fold. Assume that $X$ is smooth away from a finite number of nodal points $\{a_1, \dots, a_n\}$. Suppose that there is a projective birational morphism $f\colon Y\to X$ such that the exceptional locus $E$ of $f$ is a smooth rational curve, $f(E)=a_1$ and $Y$ is smooth away from $\{a_2, \dots, a_n\}$; i.e., $f$ is a partial small resolution of $X$ at the nodal point $a_1$. Assume that there are Cartier divisors $D_1, D_2$ on $Y$ such that, for $N_i\coloneqq R^0f_*\O_Y(D_i), i=1,2$, the following conditions are satisfied:

    (1) $D_1.E=1$, $D_2.E=-1$;

    (2) $\{N_1, N_2\}$ is a simple collection; i.e., $\dim \Hom(N_i, N_j)=\delta_{ij}$ for $1\leqslant i, j\leqslant 2$;

    (3) $H^p(X, R^0f_*\O_Y(D_i-D_j))=0$ for all $p>0$ and $1\leqslant i, j\leqslant 2$;

    Then there are locally free sheaves $F_1, F_2$ of rank $2$ on $X$ given by non-trivial extensions
    \begin{equation}\label{ncdef}
        \begin{split}
            0\to N_2\to F_1\to N_1\to 0,\\
            0\to N_1\to F_2\to N_2\to 0.
        \end{split}
    \end{equation}
    Let $F=F_1\oplus F_2$ and let $\langle N_i\rangle_{i=1}^2$ be the triangulated subcategories of $\D(X)$ generated by $N_i$. Then

    (i) $\Ext^p(F,F)=0$ for all $p>0$, $\End(F)$ is isomorphic to the path algebra $R_1$ defined by (\ref{rn}), and $F$ is flat over $R_1$;

    (ii) The functor $\Phi=-\otimes_{R_1} F\colon \D(R_1) \to \D(X)$ is fully faithful and it induces an equivalence  $\D(R_1) \cong \langle N_i\rangle_{i=1}^2$.
\end{prop}
\begin{proof}
    (i) Note that the path algebra $R_1$ is isomorphic to 
    $R=\left(
    \begin{array}{cc}
        \kk & \kk t \\
        \kk t & \kk
    \end{array}
    \right)$
    mod $t^2$ in \cite[Theorem 6.1]{kawaodp}. The same proof in \textit{loc. cit.} concludes (i).

    (ii) It follows from the proof in \cite[Theorem 5.1 (1)]{kawaodp}. Let $\Psi=\RHom(F,-)\colon \D(X)\to \D(R_1)$. Then $\Psi$ is the right adjoint of $\Phi$ and $\Psi\circ \Phi$ is the identity, which implies that $\Phi$ is fully faithful. Note that $\D(R_1)$ is a triangulated category generated by its simple modules (one for each vertex in the quiver) and their images under $\Phi$ are $N_1, N_2$. This gives the equivalence.

    One should note that the assumptions that $X$ has just \textbf{one} nodal point and $Y$ is smooth in \cite[Theorem 6.1]{kawaodp} are only needed to prove that $N_1\oplus N_2$ generates the triangulated category of singularities $\DD_{\sg}(X)$.
\end{proof}
\section{Geometry of Nodal Quintic Del Pezzo Threefolds}\label{geometry}
In this section we will describe the nodal quintic del Pezzo threefolds $X_m$ in a way that their derived categories can be understood. In short, there is a small resolution $f_m\colon Y_m\to X_m$ at a nodal point and $p_m\colon Y_m\to \P^1$ is a quadric surface fibration. 
\begin{defn}
    A quintic del Pezzo threefold is a normal integral projective threefold $X$ with at worst terminal Gorenstein singularities such that $-K_X$ is ample, divisible by $2$ and  $(-K_X/2)^3=5$.
\end{defn}
The classification indicates that the quintic del Pezzo threefolds have at worst nodal singularities and have at most $3$ nodal points. The quintic del Pezzo threefold $X_m$ with $m$ nodes ($0\leqslant m\leqslant 3$) exists and is unique up to isomorphism for each $m$ by \cite[Corollary 8.7]{prodp3fold}. Moreover, they can be realized as codimension $3$ linear sections of $\Gr(2,5)$ embedded into $\P^9$ via Pl\"ucker embedding. Let $X$ be a codimension $3$ linear section of $\Gr(2,5)$. For a generic choice of the linear section $X$ is smooth. We first give a general discussion about when $X$ is a quintic del Pezzo threefold $X_m, 0\leqslant m\leqslant 3$ and then describe the singular $X_m, 1\leqslant m \leqslant 3$ as explicit linear sections in Lemma \ref{xm}. 

Let $V_5$ be a $5$-dimensional $\kk$-vector space and let $L\subset \Lambda^2 V_5^\vee$ be a $3$-dimensional subspace where $V_5^\vee$ is the dual vector space. Denote the orthogonal complement by $L^\perp \coloneqq \ker(\Lambda^2 V_5\to L^\vee)$. Let $X=\Gr(2, V_5) \cap \P(L^\perp)$. 
\begin{enumerate}
    \item $X$ is the smooth quintic del Pezzo threefold $X_0$ if and only if $\P(L)\cap \Gr(2, V_5^\vee)=\emptyset$; see \cite[Example 6.1]{kuzhs}.
    \item Assume that $L\subset \Lambda^2 V^\vee$ is a generic subspace such that $\P(L)\cap \Gr(2, V_5^\vee)$ is a disjoint union of $m$ points for $m=1,2,3$. Then $X$ is the nodal quintic del Pezzo threefold $X_m$ with $m$ nodes. 
\end{enumerate}
In Case (2) there is a one-to-one correspondence between the points on $\P(L)\cap \Gr(2, V_5^\vee)$ and the nodal points on $X$ as we now explain. Let $H$ be a hyperplane of $\P(\Lambda^2 V)$ corresponding to a point $p_H\in \P(L)\cap \Gr(2, V_5^\vee)$. Let $\P_H$ be the singular locus of $\Gr(2, V_5) \cap H$. Represent the point $p_H$ by a $2$-dimensional subspace $A_2$ of $V_5^\vee$ and let $B_3\coloneqq\ker{(V_5\to A_2^\vee)}$. Then $\P_H \cong \Gr(2, B_3) \cong \P^2$ and $\P_H\cap X$ is the node on $X$ corresponding to $p_H$. 

\medskip
From now on, we focus on singular $X_m$ for $1 \leqslant m \leqslant 3$. Recall that $\Gr(2,5)$ is the intersection of $5$ quadrics in $\P^9$. Denote the coordinates of $\P^9=\P(\Lambda^2 V_5)$ by $\{x_{ij}\}_{1\leqslant i<j\leqslant 5}$. Then $\Gr(2,5)$ is defined by the Pl\"ucker equations $\{x_{ij}x_{kl}-x_{ik}x_{jl}+x_{il}x_{jk}=0\}$ for $1\leqslant i<j<k<l\leqslant 5$.
\begin{lemma}\label{xm}
    The quintic del Pezzo threefold $X_m$ with $m$ nodes for $1\leqslant m\leqslant 3$ can be described by the following codimension $3$ linear sections of $\Gr(2,5)$:
    \begin{enumerate}
        \item $X_1\cong \Gr(2,5)\cap \{x_{45}=x_{23}+x_{14}=x_{13}+x_{25}=0\}$ and it has one node $a_1$;
        \item $X_2 \cong \Gr(2,5)\cap \{x_{45}=x_{23}=x_{13}+x_{14}+x_{25}=0\}$ and it has two nodes $a_1, a_2$;
        \item $X_3 \cong \Gr(2,5)\cap \{x_{45}=x_{23}=x_{13}+x_{14}=0\}$ and it has three nodes $a_1, a_2, a_3$
    \end{enumerate}
    where $a_1, a_2, a_3$ are points on $\P^9$ such that all coordinates are $0$ except for $x_{12}, x_{15}, x_{25}$, respectively.
\end{lemma}
\begin{proof}
    One can check directly that the linear sections in (1)-(3) have only nodal singularities at the given points $a_i$. Hence, they are quintic del Pezzo threefolds and the isomorphisms follow from the uniqueness of quintic del Pezzo threefolds with $m$ nodes for $1\leqslant m\leqslant 3$. 
\end{proof}

We will construct a birational morphism resolving the singularities of $X_m$ at the node $a_1$. We start with a general set-up and then use Lemma \ref{xm} to apply to $X_m$. 

Let $x=[V_2]\in\Gr(2,V_5)$ be a fixed point that is represented by a $2$-dimensional subspace $V_2$ of $V_5$. Consider the natural rational map 
\begin{equation}\label{rationalmap}
    \phi\colon \Gr(2, V_5) \dashrightarrow \Gr(2, V_5/V_2)\cong \P(\Lambda^2 (V_5/V_2))\cong \P^2
\end{equation} 
sending $[W_2]$ to its image in $V_5/V_2$.  Then the base locus of $\phi$ is $\{[W_2]\in\Gr(2,V_5)\,|\, W_2\cap V_2\neq 0\}$ or equivalently the union of lines in $\Gr(2, V_5)$ that contain $x$. 
This is a cone over $\P^1\times \P^2$ with vertex $x$. Moreover, $\phi$ is a linear projection from the embedded projective tangent space $T_x \Gr(2,V_5) \cong \P^6\subset \P^9$ of $\Gr(2,V_5)$ at $x$. Let $H$ be a hyperplane of $\P(\Lambda^2 V_5)$ corresponding to a point on 
\begin{equation*}
    \Gr(2, (V_5/V_2)^\vee) \cong \P(\Lambda^2 (V_5/V_2)^\vee) \subset \Gr(2, V_5^\vee) \subset \P(\Lambda^2 V_5^\vee).
\end{equation*}
These are the hyperplanes $H$ where $T_x\Gr(2,V_5)\subset H$, or equivalently $x$ is contained in the singular locus of $\Gr(2, V_5)\cap H$. The rational map $\phi$ (\ref{rationalmap}) restricts to a linear projection
\begin{equation}\label{ratmaph}
    \phi_H\colon \Gr(2, V_5)\cap H \dashrightarrow \P^1
\end{equation}
from $\Gr(2, V_5)\cap T_x\Gr(2,V_5) \subset \Gr(2, V_5)\cap H$. Blowing up $\Gr(2, V_5)\cap H$ along the base locus $\Gr(2, V_5)\cap T_x\Gr(2,V_5)$ of $\phi_H$ gives the resolution of indeterminancy
\begin{equation}\label{resindet}
    \begin{tikzcd}[column sep = small]
        & & \P^2\times \P^1 \arrow[hook]{r} \arrow{dl} &  Y=\Gr_{\P^1}(2, \O_{\P^1}^3\oplus \O_{\P^1}(-1))\arrow{ld}[swap]{f} \arrow{rd}{p} &\\
        & \P^2 \arrow[hook]{r} & \Gr(2, V_5)\cap H \arrow[dotted]{rr}{\phi_H} & & \P^1
    \end{tikzcd}
\end{equation}
where the exceptional locus of $f$ is $\Gr_{\P^1}(2, \O_{\P^1}^3)\cong \P^2\times \P^1$ and $f$ restricted to $\Gr_{\P^1}(2, \O_{\P^1}^3)$ is a trivial $\P^1$-bundle over the singular locus $\P^2$ of $\Gr(2, V_5)\cap H$. 

Alternatively, if we represent $H=[A_2]\in \Gr(2, V_5^\vee)$ by a $2$-dimensional subspace $A_2$ of $V_5^\vee$ and let $B_3=\ker(V_5\to A_2^\vee)$, then
\begin{equation*}
    \Gr(2, V_5)\cap H \cong \{[W_2]\in \Gr(2, V_5)\,|\, W_2\cap B_3\neq 0\}.
\end{equation*}
Since $[A_2]\in \Gr(2, (V_5/V_2)^\vee)$, we have $V_2\subset B_3$. Using this description we have
\begin{equation*}
    Y\cong \{([W_2], [W_4])\in \Gr(2, V_5)\times \Gr(4, V_5)\,|\, W_2\subset W_4, B_3\subset W_4\}
\end{equation*}
and maps $f,p$ are given as
\begin{equation*}
    \begin{split}
        & f\colon Y\to \Gr(2, V_5)\cap H, \quad ([W_2], [W_4]) \mapsto [W_2],\\
        & p\colon Y\to \P^1\subset \Gr(2, V_5/V_2), \quad ([W_2], [W_4]) \mapsto [W_4/V_2].
    \end{split}
\end{equation*}

Using the descriptions of $X_m$ in Lemma \ref{xm}, we take $x=a_1$ (i.e., $V_2$ is generated by the basis $e_1, e_2$ with $e_1\wedge e_2=x_{12}$) and $H=\{x_{45}=0\}$. We get that $X_m\subset H$ for $1\leqslant m\leqslant 3$, the target $\P(\Lambda^2 (V_5/V_2))$ of the map $\phi$ (\ref{rationalmap}) is $\P^2_{x_{34}, x_{35}, x_{45}}$ and the singular locus of $\Gr(2,5)\cap H$ is $\P^2_{x_{12}, x_{13}, x_{23}}$.

\begin{prop}\label{setup}
    Let $\phi_m\colon X_m\dashrightarrow \P^1$ be the restriction of the rational map $\phi_H$ (\ref{ratmaph}) for $1\leqslant m\leqslant 3$. Let $Y_m=f^{-1}(X_m)$. Let $f_m=f|_{Y_m}\colon Y_m\to X_m$ and $p_m=p|_{Y_m}\colon Y_m\to \P^1$ be the respective restrictions of maps in (\ref{resindet}). By construction we have $\phi_m=p_m\circ f_m^{-1}$ and the following descriptions hold.

    (i) The map $f_m\colon Y_m\to X_m$ is a birational morphism that resolves the singularity of a nodal point $a_1\in X_m$ and contracts a smooth rational curve $E$ on $Y_m$ to $a_1$, and the map $p_m\colon Y_m\to \P^1$ is a quadric surface fibration where $E$ is a smooth section (consists of smooth points on the fibers). In particular, $Y_m$ has $m-1$ nodal points and is Gorenstein. The fibers of $p_m, m=1,2$ are quadrics of corank at most $1$ (smooth or a cone over a smooth quadric) and $p_3$ has a fiber of corank $2$.

    (ii) Let $\E=\O_{\P^1}\oplus \O_{\P^1}(-1)^3$ and $\cL=\O_{\P^1}(-1)$. Let $\pi\colon \P(\E)\to \P^1$ be the projection. Then the total space $Y_m$ of the quadric surface fibration $p_m$ is the zero locus of a global section
    \begin{equation*}
        \sigma_m \in \Gamma(\P(\E), \O_{\P(\E)/\P^1}(2)\otimes \pi^*\cL) \cong \Gamma(\P^1, \Sym^2(\E^\vee) \otimes \cL)
    \end{equation*}
    on $\P(\E)$ where $\E^\vee$ is the dual of $\E$ and $\Sym^2$ is the second symmetric product. 

    (iii) The exceptional locus $E$ of $f_m\colon Y_m\to X_m$ is the projectivization of $\O_{\P^1}\subset \E$ and $E$ is a $(-1,-1)$-curve on $Y_m$; i.e., the normal bundle $N_{E/Y_m}$ is $\O_E(-1)^2$.

    (iv) We have
    \begin{equation*}
        \O_{Y_m/\P^1}(1) \cong f_m^*\O_{X_m}(1)
    \end{equation*}
    where $\O_{Y_m/\P^1}(1)$ is the restriction of $\O_{\P(\E)/\P^1}(1)$ to $Y_m\subset \P(\E)$ and $\O_{X_m}(1)$ is the restriction of $\O_{\P^6}(1)$ to $X_m\subset \P^6$. In particular, we get $\O_{Y_m/\P^1}(1)|_E\cong \O_E$ and
    \begin{equation*}
        \omega_{Y_m} \cong f_m^*\omega_{X_m} \cong \O_{Y_m/\P^1}(-2)
    \end{equation*}
    (the dualizing sheaves $\omega_{Y_m}, \omega_{X_m}$ are line bundles because $Y_m, X_m$ are Gorenstein).

    (v) Let $L_m$ be the Hilbert scheme of lines on $X_m$ that contain $a_1$. Then $L_m$ is a nodal curve of arithmetic genus $0$ and degree $3$ on $\P^3$. More specifically, $L_1$ is the twisted cubic curve, $L_2$ is a chain of two $\P^1$'s, and $L_3$ is a chain of three $\P^1$'s. Moreover, we can embed $L_m$ into $ X_m$ such that $a_1\notin L_m$, $\{a_2\}$ is the singular point of $L_2$ and $\{a_2, a_3\}$ are the singular points of $L_3$.
\end{prop}
\begin{proof}
    We can use the explicit equations in Lemma \ref{xm} for the argument.

    (i) We first note that the ideal of $\Gr(2,5)\cap H = \Gr(2,5)\cap \{x_{45}=0\}$ in $\P^8$ is generated by
    \begin{equation}\label{minor}
        \text{the three $2$-by-$2$ minors making}\quad
        \rank \left(
        \begin{array}{ccc}
            x_{14} & x_{24} & x_{34} \\
            x_{15} & x_{25} & x_{35} 
        \end{array}
        \right) \leqslant 1
    \end{equation}
    and 
    \begin{equation*}
        \left(
        \begin{array}{ccc}
            x_{14} & x_{24} & x_{34} \\
            x_{15} & x_{25} & x_{35} 
        \end{array}
        \right)
        \left(
        \begin{array}{r}
            x_{23}\\
            -x_{13}\\
            x_{12}
        \end{array}
        \right)
    \end{equation*}
    We also have that the base locus $\Bs(\phi_m)$ of $\phi_m$ is given by $X_m\cap\{x_{34}=x_{35}=0\}$. Then $Y_m$, which is the blow up of $X_m$ along $\Bs(\phi_m)$, can be described explicitly as
    \begin{equation}\label{ym}
        Y_m \cong \Proj \frac{\kk[x_{12}, x_{i4}, x_{i5}, 1\leqslant i\leqslant 3][u, v]}{(vx_{i4}-ux_{i5}, 1\leqslant i\leqslant 3, q_{m,1}, q_{m,2})}
    \end{equation}
    where $q_{m,1}, q_{m,2}$ are quadratic forms
    \begin{itemize}
        \item $q_{1,1}=x_{12}x_{34}+x_{24}x_{25}-x_{14}^2$, $q_{1,2}=x_{12}x_{35}+x_{25}^2-x_{14}x_{15}$;
        \item $q_{2,1}=x_{12}x_{34}+(x_{14}+x_{25})x_{24}$, $q_{2,2}=x_{12}x_{35}+(x_{14}+x_{25})x_{25}$;
        \item $q_{3,1}=x_{12}x_{34}+x_{14}x_{24}$, $q_{3,2}=x_{12}x_{35}+x_{14}x_{25}$.
    \end{itemize}
    The exceptional locus $E$ is defined by $\{x_{i4}=x_{i5}=0, 1\leqslant i\leqslant 3\}$. Then $E\cong \P^1$ and $f_m \colon Y_m\to X_m$ maps $E$ to the node $a_1$. 

    From (\ref{ym}) we get that $Y_m$ is a subscheme of
    \begin{equation}\label{projb}
        P=\Proj \frac{\kk[x_{12}, x_{i4}, x_{i5}, 1\leqslant i\leqslant 3][u, v]}{(vx_{i4}-ux_{i5}, 1\leqslant i\leqslant 3)},
    \end{equation}
    which is the zero locus of a regular section of $\O_{\P^6\times \P^1}(1,1)^3$. One can easily compute that $P\cong \P(\E)$ where $\E= \O_{\P^1}\oplus \O_{\P^1}(-1)^3$.

    Let $\pi\colon \P(\E)\to \P^1$ be the projection. Let $U=\pi^{-1}(\{u=1\})$ and $V=\pi^{-1}(\{v=1\})$ be the open cover of $\P(\E)$. Then inside $U$ and $V$, respectively, we have
    \begin{equation}\label{quad-eq}
        \begin{split}
            & Y_1\cap U = \{x_{12}x_{34}+vx_{24}^2-x_{14}^2=0\}, \quad Y_1\cap V=\{x_{12}x_{35}+x_{25}^2-ux_{15}^2=0\};\\
            & Y_2\cap U=\{x_{12}x_{34}+x_{14}x_{24}+vx_{24}^2=0\}, \quad Y_2\cap V=\{x_{12}x_{35}+ux_{15}x_{25}+x_{25}^2=0\};\\
            & Y_3\cap U=\{x_{12}x_{34}+x_{14}x_{24}=0\}, \quad Y_3\cap V=\{x_{12}x_{35}+ ux_{15}x_{25}=0\}.
        \end{split}  
    \end{equation}

    One easily checks that $E$ is a smooth section from the equations above. Moreover, $p_1$ has singular fibers over the points $\{u=0\}$, $\{v=0\}$ and the fibers are of corank $1$; $p_2$ has a singular fiber over the point $\{u=0\}$ and the fiber is of corank $1$; $p_3$ has a singular fiber over the point $\{u=0\}$ and the fiber is of corank $2$. This concludes (i).

    \medskip
    For the proof of (ii)-(iv) we observe from the proof of (i) that $Y_m$ is the zero locus of a global section 
    \begin{equation}\label{ymsec}
        \sigma_m \in \Gamma(\P(\E), \O_{\P(\E)/\P^1}(2) \otimes \pi^* M_m) \cong \Gamma(\P^1, \Sym^2(\E^\vee) \otimes M_m)
    \end{equation}
    for some line bundle $M_m$ on $\P^1$. This implies that $p_{m*}\O_{Y_m/\P^1}(1) \cong \E^\vee \cong \O_{\P^1} \oplus \O_{\P^1}(1)^3$.

    (iv) Since each fiber of $p_m$ is a quadric surface contained in $X_m$ that passes through $a_1$, we have $\O_{Y_m/\P^1}(1) \cong f_m^* \O_{X_m}(1) \otimes p_m^* N_m$ for some line bundle $N_m$ on $\P^1$. Since
    \begin{equation*}
        h^0(X_m, \O_{X_m}(1)) = 7 = h^0(Y_m, \O_{Y_m/\P^1}(1)),
    \end{equation*}
    we get $N_m\cong\O_{\P^1}$. Since $f_m\colon Y_m\to X_m$ is a small contraction and $X_m$ is a Fano variety of index $2$, 
    \begin{equation*}
        \omega_{Y_m} \cong f_m^*\omega_{X_m} \cong f_m^*\O_{X_m}(-2) \cong \O_{Y_m/\P^1}(-2).
    \end{equation*}

    (ii) It remains to determine the line bundle $M_m$ in (\ref{ymsec}). We achieve this by computing the dualizing sheaf $\omega_{Y_m}$. The adjunction formula gives
    \begin{equation*}
        \begin{array}{rl}
            \omega_{Y_m} & \cong (\omega_{\P(\E)} + Y_m)\,|_{Y_m}\\ 
            & \cong (\O_{\P(\E)/\P^1}(-4)\otimes \pi^*\det(\E^\vee)\otimes \pi^*\omega_{\P^1} \otimes \O_{\P(\E)/\P^1}(2)\otimes \pi^* M_m)\,|_{Y_m}\\
            & \cong \O_{Y_m/\P^1}(-2)\otimes p_m^* (\O_{\P^1}(1)\otimes M_m).\\
        \end{array}
    \end{equation*}
    Hence, using (iv) we get $M_m \cong \O_{\P^1}(-1)$.

    (iii) It is easy to see from the proof of (i) that $E$ is the projectivization of $\O_{\P^1}\subset \E$ and the normal bundle $N_{E/\P(\E)}$ is $\O_E(-1)^3$. In terms of the normal bundle $N_{E/Y_m}$, we make use of the short exact sequence of normal bundles
    \begin{equation*}
        0\to N_{E/Y_m}\to N_{E/\P(\E)} \to N_{Y_m/\P(\E)}|_E \to 0.
    \end{equation*}
    Since $\O_{Y_m/\P^1}(1)|_E\cong \O_E$ by (iv) and $E$ is a section of $p_m$, we get from (ii) that
    \begin{equation*}
        N_{Y_m/\P(\E)}|_E \cong (\O_{Y_m/\P^1}(2)\otimes p_m^*\O_{\P^1}(-1))|_E \cong \O_E(-1).
    \end{equation*}
    This implies that $N_{E/Y_m}\cong \O_E(-1)^2$.

    \medskip
    (v) We mentioned after the construction of the map $\phi$ (\ref{rationalmap}) that the base locus $\Bs(\phi)$ consists of lines on $\Gr(2,5)$ that pass through $x=a_1$. Then $\Bs(\phi_m)$ as a linear section of $\Bs(\phi)$ consists of lines on $X_m$ that pass through $a_1$. That is, $\Bs(\phi_m)$ is the cone over $L_m$ with the vertex $a_1$ and $L_m$ can be identified with the projection of $\Bs(\phi_m)$ from $a_1$. Explicitly, $\Bs(\phi_m)=X_m\cap\{x_{34}=x_{35}=x_{45}=0\}$ and $L_m \cong \Bs(\phi_m)\cap \{x_{12}=0\}$. To describe $L_m$ one notes that $\Gr(2,5)\cap \{x_{12}=x_{34}=x_{35}=x_{45}=0\}$ is isomorphic to $\P^1\times \P^2$ and the Segre embedding $\P^1\times \P^2\hookrightarrow \P^5$ is of degree $3$. Therefore, $L_m$ is obtained by cutting $\P^1\times \P^2$ with the respective hyperplane sections for $X_m$.
\end{proof}

Recall from \S\ref{sec-notation} that we denoted
\begin{equation*}
    \begin{split}
        & \B_{m,0} \cong \O_{\P^1}\oplus (\Lambda^2\E \otimes \cL^\vee) \oplus (\Lambda^4\E \otimes (\cL^2)^\vee) \cong \O_{\P^1}^4 \oplus \O_{\P^1}(-1)^4,\\
        & \B_{m,1} \cong \E \oplus (\Lambda^3\E \otimes \cL^\vee) \cong \O_{\P^1} \oplus \O_{\P^1}(-1)^6 \oplus \O_{\P^1}(-2)
    \end{split}
\end{equation*}
as even and odd parts of the Clifford algebra of the quadric surface fibration $p_m\colon Y_m\to \P^1$, respectively. Then
\begin{equation*}
    \Z_m= \O_{\P^1}\oplus \Lambda^4\E \otimes (\cL^2)^\vee \cong \O_{\P^1}\oplus \O_{\P^1}(-1)
\end{equation*}
is a central subalgebra of $\B_{m,0}$ and the map $g_m\colon C_m = \Spec_{\P^1}(\Z_m) \to \P^1$ is the double cover ramified along the degeneration locus of $p_m$. We also denoted $\wt{\B_{m,0}}$ as the sheaf of algebras on $C_m$ when one regards $\B_{m,0}$ as a sheaf of algebras over $\Z_m$.

\begin{lemma}\label{doublecover}
    (i) $C_1\cong \P^1$ and $C_2 \cong C_3$ is a chain of two $\P^1$'s. The map $g_1\colon C_1\to \P^1$ is a double cover ramified at $[1: 0], [0: 1]\in \P^1$ and the map $g_i \colon C_i\to \P^1$ is a double cover ramified at the double point supported at $[0:1]\in\P^1$ for $i=2,3$.

    (ii) Let $H(p_m)$ be the Hilbert scheme of lines on the fibers of $p_m\colon Y_m\to\P^1$ for $m=1,2$. Let
    \begin{equation*}
        \cV_m=\left\{
        \begin{array}{ll}
            \O_{C_1}\oplus \O_{C_1}(-1), & m=1\\
            \O_{C_2}\oplus \O_{C_2}\{-1,0\}, & m=2\\
        \end{array}
        \right.
    \end{equation*}
    where $\O_{C_2}\{a,b\}$ is the line bundle whose restriction to the first $\P^1$ is $\O_{\P^1}(a)$ and to the second $\P^1$ is $\O_{\P^1}(b)$. Then $\wt{\B_{m,0}}\cong \sEnd(\cV_m)$ and $H(p_m)\cong \P(\cV_m)$ for $m=1,2$.
\end{lemma}
\begin{proof}
    (i) By Proposition \ref{setup} (ii) the quadric surface fibration $p_m\colon Y_m\to \P^1$ corresponds to a symmetric bilinear form $b_m\colon \Sym^2(\E)\to \cL$. The form $b_m$ induces a morphism $\E\to \sHom(\E, \cL)$ whose determinant gives a global section of
    \begin{equation*}
        \Gamma(\P^1, \det(\E^\vee)^2\otimes \cL^4)\cong \Gamma(\P^1, \O_{\P^1}(2)).
    \end{equation*}
    Denote the global section by $\det(b_m)$. The degeneration locus of $p_m$ is given by $\{\det(b_m)=0\}$ and the maps $g_m\colon C_m=\Spec_{\P^1}(\Z_m)\to \P^1$ are locally defined by $\Spec \O_{\P^1}[d]/(d^2-\det(b_m))$. The equations (\ref{quad-eq}) of $Y_m$ indicate that
    \begin{equation*}
        \det(b_1)=\lambda_1 uv, \quad \det(b_2)=\lambda_2 u^2, \quad \det(b_3)= \lambda_3 u^2
    \end{equation*}
    for some non-zero $\lambda_m\in \kk$. The result then follows.

    (ii) Proposition \ref{setup} (i) states that the fibers of $p_m, m=1,2$ are quadrics of corank $\leqslant 1$. By \cite[Proposition 3.13]{kuzqfib} $\wt{\B_{m,0}}$ is a sheaf of Azumaya algebras on $C_m$. The natural morphism $H(p_m)\to \P^1$ factors as the composition of a $\P^1$-fibration $H(p_m)\to C_m$ followed by $g_m\colon C_m\to \P^1$. In fact, $H(p_m)$ is the Severi-Brauer scheme or the $\P^1$-fibration over $C_m$ corresponding to the sheaf of Azumaya algebras $\wt{\B_{m,0}}$. Furthermore, $H(p_m)\to C_m$ has a section given by the lines on the fibers of $p_m$ intersecting the section $E$. Then $\wt{\B_{m,0}}$ is Brauer trivial; i.e., it is $\sEnd(\cV_m)$ for some rank $2$ vector bundle $\cV_m$ on $C_m$ and $H(p_m) \cong \P(\cV_m)$.

    It remains to determine $\cV_m$. Every indecomposable vector bundle on a chain of $\P^1$'s is a line bundle. This fact is used in \cite[\S2]{burratsingc} and a proof can be found in Corollary 6.2 of \cite{drozdvb}. Hence, up to tensoring by a line bundle, we can assume 
    \begin{equation*}
        \cV_m=\O_{C_m}\oplus \cL_m
    \end{equation*}
    for a line bundle $\cL_m$ on $C_m$. 

    One observes from (\ref{quad-eq}) that $Y_1, Y_2$ are defined by the type of quadratic equations in Example \ref{qfex}. Adopting the notations in this example, we can set the local sections of direct summands of $\E=\O_{\P^1}\oplus \O_{\P^1}(-1)^3$ as below.
    \begin{equation}\label{basis}
        v_1, v_2, e_3, e_4 \quad \text{are local sections of} \quad \O_{\P^1}, \O_{\P^1}(-1), \O_{\P^1}(-1), \O_{\P^1}(-1), \; \text{respectively.}
    \end{equation}
    The local sections of $\sEnd(\cV_m)$ can be obtained from (\ref{matrix}). In particular, 
    \begin{equation*}
        \begin{split}
            & \sHom(\O_{C_m}, \cL_m)\cong \cL_m \quad \text{has a local section} \quad v_2v_3,\\
            & \sHom(\cL_m, \O_{C_m}) \cong \cL_m^\vee \quad \text{has a local section} \quad v_1v_3.
        \end{split}
    \end{equation*}
    From (\ref{actionbycenter}) we get that 
    \begin{equation}\label{vb12}
        \begin{split}
            & g_{m*}\cL_m \quad \text{is a rank $2$ vector bundle with local sections} \quad v_2e_3, v_2e_4,\\
            & g_{m*}\cL_m^\vee \quad \text{is a rank $2$ vector bundle with local sections} \quad v_1e_3, v_1e_4.
        \end{split}
    \end{equation}
    Note that $\B_{m,0}\cong g_{m*}\sEnd(\cV_m)$. Hence, $g_{m*}\cL_m$, $g_{m*}\cL_m^\vee$ are subbundles of $\Lambda^2\E\otimes \cL^\vee$ with respective local sections. Since $\cL=\O_{\P^1}(-1)$, we get
    \begin{equation*}
        g_{m*}\cL_m\cong \O_{\P^1}(-1)^2, \quad g_{m*}\cL_m^\vee\cong \O_{\P^1}^2.
    \end{equation*}
    Therefore, $H^\bullet(C_m, \cL_m)=0$ and $H^\bullet(C_m, \cL_m^\vee)=\kk^2$. This implies that $\cL_1\cong \O_{C_1}(-1)$, and $\cL_2\cong \O_{C_2}\{-1,0\}$ or $\O_{C_2}\{0,-1\}$. Up to the involution of $g_2$, we can make $\cL_2 \cong \O_{C_2}\{-1,0\}$.
\end{proof}
\section{Derived Categories of Nodal Quintic Del Pezzo Threefolds}\label{derivedcat}
The goal of the section is to construct a Kawamata type semiorthogonal decomposition of $\D(X_m)$ using the birational morphism $f_m\colon Y_m\to X_m$ and the quadric surface fibration $p_m\colon Y_m\to \P^1$ constructed in \S\ref{geometry}. There are three steps towards this construction.

Firstly, we take the SOD of $\D(Y_m)$ constructed by Kuznetsov \cite{kuzqfib} and try to understand the nontrivial subcategory $\D(\P^1, \B_{m,0})$. The majority of the work in this step goes into the $3$-nodal case $X_3$ and the result is given in Proposition \ref{3rdcase}. 

Secondly, we work out the objects generating each component of the SOD of $\D(Y_m)$ and perform a series of mutations to obtain a new SOD of $\D(Y_m)$ that can be descended to a SOD of $\D(X_m)$; see Corollary \ref{ymsodnew}. This is where the spinor sheaf associated with the smooth section $E$ appears; see Lemma \ref{se}. 

Lastly, Proposition \ref{kawaprop} gives a generalization of the conditions proposed by Kawamata in \cite{kawaodp} for the descent of derived categories, and Lemma \ref{dbr2} checks that these conditions are satisfied. Together with Proposition \ref{pushforward} and \ref{indsod}, we can prove the main result Theorem \ref{mainthm} of the paper.

\medskip
Recall maps $p_m\colon Y_m\to \P^1$, $\pi\colon \P(\E)\to \P^1$ and $i_m\colon Y_m\hookrightarrow \P(\E)$ ($p_m = \pi \circ i_m$) for $m=1,2,3$.
\begin{thm}[{\cite[Theorem 4.2]{kuzqfib}}]
    There are $\P^1$-linear admissible semiorthogonal decompositions
    \begin{equation}\label{ymsod}
        \D(Y_m)=\langle \Phi_{m,l} (\D(\P^1, \B_{m,0})), p_m^*\D(\P^1), p_m^*\D(\P^1)\otimes\O_{Y_m/\P^1}(1) \rangle
    \end{equation}
    where $\D(\P^1, \B_{m,0})$ is the derived category of coherent sheaves on $\P^1$ with right $\B_{m,0}$-module structures. Moreover, $\Phi_{m,l}, l\in\mathbb{Z}$ is the embedding functor 
    \begin{equation}\label{fmphi}
        \Phi_{m, l}= p_m^* (-)\otimes^{\LL}_{p_m^*\B_{m,0}} K_{m, l}\colon \D(\P^1, \B_{m,0})\hookrightarrow \D(Y_m)
    \end{equation}
    where $K_{m, l}$ is a rank $4$ vector bundles on $Y_m$ with left $\B_{m,0}$-module structures defined by the short exact sequence
    \begin{equation}\label{kerses}
        0\to \O_{\P(\E)/\P^1}(-2)\otimes \pi^*(\B_{m,0}\otimes \O_{\P^1}(l)) \xrightarrow{\delta_m} \O_{\P(\E)/\P^1}(-1)\otimes \pi^*(\B_{m,1} \otimes \O_{\P^1}(l)) \to i_{m*} K_{m, l}\to 0. 
    \end{equation}
    Here $\delta_m$ is the multiplication induced by $\delta$ from the right (recall $\delta$ from \S\ref{sec-notation}) and $\delta_m \cdot \delta_m = \sigma_m$ where $\sigma_m$ is the global section defining the total space $Y_m$ of the quadric fibration. 
\end{thm}
\begin{rem}
   (1) The SOD (\ref{ymsod}) is admissible because $Y_m$ is Gorenstein projective by Proposition \ref{setup} (i) and we can apply Lemma \ref{adsodlem}.

    (2) The map $\delta_m$ in (\ref{kerses}) is defined in the same way as maps in (\ref{spinor}). Thus, up to the twist by $\O_{Y_m/\P^1}(-1)$, the bundle $K_{m,0}$ can be regarded as the spinor sheaf $\cS_0$ associated with the zero isotropic subbundle on $Y_m$. 

    (3) There are equivalences
    \begin{equation}\label{gmeq}
        g_{m*}\colon \D(C_m, \wt{\B_{m,0}}) \xrightarrow{\cong} \D(\P^1, \B_{m,0})
    \end{equation}
    by the definition of $\wt{\B_{m,0}}$. According to Lemma \ref{doublecover} (ii), $\wt{\B_{m,0}}\cong \sEnd(\cV_m)$ for $m=1,2$ are trivial Azumaya algebras. There are equivalences
    \begin{equation*}
        -\otimes \cV_m^\vee\colon \D(C_m) \xrightarrow{\cong} \D(C_m, \wt{\B_{m,0}}), \quad m=1,2.
    \end{equation*} 
\end{rem}
Since $p_3\colon Y_3\to \P^1$ has a fiber of corank $2$, $\wt{\B_{3,0}}$ is not an Azumaya algebra. It is not clear a priori what $\D(C_3, \wt{\B_{3,0}})$ looks like. We study this subcategory below.

\medskip
Let $M$ be a chain of three $\P^1$'s and let $N$ be a chain of two $\P^1$'s. Let $M_i, i=1,2,3$ be the $i$-th component of $M$ and let $N_j, j=1,2$ be the $j$-th component of $N$. Let $h\colon M\to N$ be the map that contracts $M_2$ to the point $N_1\cap N_2$. That is, $h|_{M_1}\colon M_1\to N_1, h|_{M_3}\colon M_3\to N_2$ are identities of $\P^1$ and $h|_{M_2}\colon M_2\to N_1\cap N_2$ is the constant map.

Recall from \S\ref{sec-notation} the notation $\O_M\{d_1,\dots, d_3\}$ for line bundles on $M$. 

\begin{lemma}\label{3to2}
    We have $h_*\O_M\{0,-1,0\}\cong \O_{N_1}(-1)\oplus \O_{N_2}(-1)$ and $h_*\O_M\{0,1,0\}\cong \O_{N_1}\oplus \O_{N_2}$.
\end{lemma}
\begin{proof}
    Let $x\in M_2$ be a smooth point and let $y=N_1\cap N_2$ be the intersection point. There is a short exact sequence
    \begin{equation}\label{ses}
        0\to \O_{M}\{0,-1,0\} \to \O_M \to \O_x \to 0.
    \end{equation}
    Note that $h_*\O_M \cong \O_N$. Applying $h_*$, we have
    \begin{equation*}
        0\to R^0h_*\O_{M}\{0,-1,0\} \to \O_N \to \O_y \to 0
    \end{equation*}
    and $R^ih_*\O_{M}\{0,-1,0\}=0$ for $i>0$. Thus, 
    \begin{equation*}
        h_*\O_{M}\{0,-1,0\}\cong R^0h_*\O_{M}\{0,-1,0\}\cong \O_{N_1}(-1)\oplus \O_{N_2}(-1)
    \end{equation*}
    because it is the ideal of $y$. Tensoring the sequence (\ref{ses}) by $\O_M\{0,1,0\}$ and applying $h_*$, we get
    \begin{equation*}
        0\to \O_N\to R^0h_*\O_{M}\{0,1,0\} \to \O_y \to 0
    \end{equation*}
    and $R^ih_*\O_{M}\{0,1,0\}$ for $i>0$. 

    Let $f\colon \wt{M}\to M$ be the normalization of $M$. Then $\wt{M}\cong \bigsqcup_{i=1}^3 M_i$ and there is a short exact sequence
    \begin{equation*}
        0 \to \O_M\{0,1,0\} \to f_*f^* \O_M\{0,1,0\} \to \O_{x_1}\oplus \O_{x_2} \to 0
    \end{equation*}
    where $x_1, x_2$ are the nodal points of $M$. Note that $f^* \O_M\{0,1,0\} \cong \O_{M_1} \oplus \O_{M_2}(1) \oplus \O_{M_3}$. Applying $h_*$ to the above sequence, we get
    \begin{equation*}
        0\to R^0h_*\O_{M}\{0,1,0\} \to \O_{N_1}\oplus \O_y^2 \oplus \O_{N_2} \to \O_y^2 \to 0
    \end{equation*}
    where the second map is given by natural surjections $\O_{N_i}\to \O_y$ and identities $\O_y\to \O_y$. This implies that
    \begin{equation*}
        h_*\O_M\{0,1,0\}\cong R^0h_*\O_{M}\{0,1,0\} \cong  \O_{N_1}\oplus \O_{N_2}.
    \end{equation*}
\end{proof}

We can embed $N$ into $\P^2$. Let $o\in N\subset \P^2$ be the nodal point of $N$ and let $f\colon \Bl_o \P^2\to \P^2$ be the blow-up of $\P^2$ at $o$. Then $M\cong f^{-1}(N)$ is the total transform of $N$ and $h\colon M\to N$ is the restriction of $f$. This gives the commutative diagram
\begin{equation}\label{blow-up}
    \begin{tikzcd}
        M=f^{-1}(N) \arrow[hook]{r}\arrow{d}{h} & \Bl_o \P^2 \arrow{d}{f}\\
        N \arrow[hook]{r} & \P^2.
    \end{tikzcd}
\end{equation}

Recall from Proposition \ref{setup} (v) that $L_m$ is the Hilbert scheme of lines on $X_m$ containing $a_1$ and $L_m$ can be embedded into $X_m$ with image disjoint from $a_1$. Recall that $g_m\colon C_m\to \P^1$ are double covers ramified along the degeneration locus of $p_m$.

Let $h_m\colon L_m\cong f_m^{-1}(L_m)\hookrightarrow Y_m \xrightarrow{p_m} \P^1$. By Proposition \ref{setup} (v) and Lemma \ref{doublecover} (i) we have $L_m\cong C_m$ for $m=1,2$ and $L_3\cong M, C_3\cong N$. In addition, we get
\begin{equation}\label{hm}
    h_m=
    \left\{
    \begin{array}{ll}
        g_m, & m=1,2\\
        g_3\circ h, & m=3
    \end{array}
    \right.
\end{equation}
where $h$ is the map in (\ref{blow-up}).

Since $\Bl_o \P^2\subset \P^2\times\P^1$, we have an embedding $j\colon L_3\hookrightarrow \P^1_{C_3}$ and commutative diagrams

\begin{equation}\label{l3}
    \begin{tikzcd}
        L_3 \arrow[hook]{r}{j}\arrow[swap]{rd}{h} & \P^1_{C_3} \arrow{d}{\pi_3}\\
        & C_3,
    \end{tikzcd}
    \qquad
    \begin{tikzcd}
        L_3 \arrow{r}{h}\arrow[swap]{rd}{h_3} & C_3 \arrow{d}{g_3}\\
        & \P^1
    \end{tikzcd}    
\end{equation}
where $\pi_3$ is the projection map.

Recall $\E=\O_{\P^1}\oplus \O_{\P^1}(-1)^3$ and $\cL=\O_{\P^1}(-1)$.
\begin{prop}\label{3rdcase}
    Let $\cL_3=\O_{L_3}\{0,-1,0\}$ and $\cV_3=\O_{L_3}\oplus \cL_3$. Let $h\colon L_3\to C_3$ be the map that contracts the middle $\P^1$ to the node of $C_3$. Then

    (i) $h_*\sEnd (\cV_3)\cong \wt{\B_{3,0}}$ as sheaves of $\O_{C_3}$-algebras;

    (ii) $h_*\colon \D(L_3, \sEnd (\cV_3))\xrightarrow{\cong} \D(C_3, \wt{\B_{3,0}})$ is an equivalence.
\end{prop}
\begin{proof}
    (ii) We first prove (ii) assuming (i). Since $h$ can be obtained as the restriction of the blow-up map $f$ in (\ref{blow-up}) as well as the projection map $\pi_3$ in (\ref{l3}), we get $\det(\cV_3)\cong\cL_3\cong j^*\O_{\pi_3}(-1)$, which is relative very anti-ample. Thus, $\cV_3^\vee$ is relative base-point free and $\det(\cV_3^\vee)$ is relative ample. Moreover, $\cV_3^\vee$ has a direct summand $\O_{\cL_3}$ and $H^1(L_3, \cV_3)=0$. This implies that $\cV_3$ is a local projective generator in ${}^0\mathrm{Per}(L_3/C_3)$ by \cite[Proposition 3.2.7]{vdb-flop} and $h_*$ is an equivalence by \cite[Proposition 3.3.1]{vdb-flop}.

    (i) We see from (\ref{quad-eq}) that similarly to cases of $Y_1, Y_2$, quadratic equations of $Y_3$ can also be studied using Example \ref{qfex}. We will adopt the notations in this example and apply similar arguments used in Lemma \ref{doublecover} (ii).

    If we view $q(x)$ as the $C_3$-family of quadratic equations obtained from pulling back $p_3\colon Y_3 \to \P^1$ along $g_3\colon C_3\to \P^1$, then for every vector $v$ we can regard $q(v)\colon C_3\to \kk$ as a function on closed points of $C_3$. Denote the node of $C_3$ by $o$. Note that the fiber of $Y_3\times_{\P^1} C_3\to C_3$ has corank $2$ at $o\in C_3$ and is smooth otherwise. Thus, for $i=3,4$ and $x\in C_3$, we have
    \begin{equation*}
        q(e_i)(x)\neq 0 \; \text{if} \; x\neq o, \quad q(e_i)(o)=0.
    \end{equation*}

    Since $q'(x_3, x_4)=0$ at $o\in C_3$, we get from (\ref{corank2}) that $\wt{\B_{3,0}}$ is generated by local sections $v_1v_2, v_1e_3, v_1e_4, v_2v_1, v_2e_3, v_2e_4$. Its algebra structure is described in the list below. The list only includes the multiplication of any two of these local sections that may be non-zero.
    \begin{equation*}
        \left\{
        \begin{array}{l}
             (v_1v_2)^2=v_1v_2, \; (v_1v_2)(v_1e_i)=v_1e_i,\\
             (v_2v_1)^2=v_2v_1,\; (v_2v_1)(v_2e_i)= v_2e_i,\\
             (v_1e_i)(v_2e_i)=-q(e_i)v_1v_2, \; (v_2e_i)(v_1e_i)=-q(e_i)v_2v_1,
        \end{array}
        \right. \quad i=3,4.
    \end{equation*}

    On the other hand, we will use Lemma \ref{3to2} to study the algebra structure of $h_* \sEnd (\cV_3)$. Remember $L_3=M$ and $C_3=N$. We take local sections of $h_* \sEnd (\cV_3)$ as below.  
    \begin{equation}\label{end-sec}
        \begin{split}
            & h_*\sHom(\O_{L_3}, \O_{L_3})\cong \O_N \quad \text{has a section} \; \id_{\O_{L_3}},\\
            & h_*\sHom(\cL_3, \cL_3)\cong \O_N \quad \text{has a section} \; \id_{\cL_3}, \\
            & h_*\sHom(\cL_3, \O_{L_3})\cong h_*\O_M\{0,1,0\}\cong \O_{N_1}\oplus \O_{N_2} \quad \text{has a section} \;  s_k, \quad k=1,2
        \end{split}
    \end{equation}
    such that $s_k$ is the constant function on $N_k$ with value $1$. There are short exact sequences
    \begin{equation*}
        \begin{split}
            & 0\to \O_{N_1}(-1)\xrightarrow{a_1} \O_N \xrightarrow{s_2} \O_{N_2}\to 0,\\
            & 0\to \O_{N_2}(-1)\xrightarrow{a_2} \O_N \xrightarrow{s_1} \O_{N_1}\to 0.
            \end{split}
    \end{equation*}
    Thus, for $k=1,2$ we can consider $a_k$ above as a local section of the ideal sheaf $\O_{N_k}(-1)$ and get
    \begin{equation*}
        N_2=\{a_1=0\}\subset N, \quad N_1=\{a_2=0\} \subset N.
    \end{equation*}
    This means that
    \begin{equation*}
        h_*\sHom(\O_{L_3}, \cL_3)\cong h_*\O_M\{0,-1,0\}\cong \O_{N_1}(-1)\oplus \O_{N_2}(-1) \; \text{has a local section } \;  a_k, \; k=1,2.
    \end{equation*}
    There are relations
    \begin{equation*}
        s_k\circ a_k = \delta_o \id_{\O_{L_3}}, \quad a_k\circ s_k =  \delta_o \id_{\cL_3}, \quad k=1,2
    \end{equation*}
    where $\delta_o\colon C_3\to \kk$ is a function satisfying $\delta_o(x)\neq 0$ if $x\neq o$ and $\delta_o(o)=0$.

    Now we will construct an algebra homomorphism from $\wt{\B_{3,0}}$ to $h_* \sEnd (\cV_3)$. Let 
    \begin{equation*}
        \beta\colon \wt{\B_{3,0}}\to h_* \sEnd (\cV_3)
    \end{equation*}
    be the $\O_{C_3}$-homomorphism such that
    \begin{equation}\label{beta-map}
        \beta(v_1v_2)=\id_{\O_{L_3}}, \quad \beta(v_1e_i)=s_{i-2}, \quad \beta(v_2e_i)=a_{i-2}, \quad i=3,4.
    \end{equation}
    Note that $v_2v_1=1-v_1v_2$ and this implies $\beta(v_2v_1)=\id_{\cL_3}$. From the discussion of the algebra structures on both sides, we get that $\beta$ is an algebra homomorphism up to multiplying $a_1, a_2$ with a nowhere zero function on $C_3$. 

    We claim that $\beta$ is an isomorphism. Note that
    \begin{equation*}
        \begin{split}
            & g_{3*}\wt{\B_{3,0}}\cong \B_{3,0} \cong \O_{\P^1}\oplus (\Lambda^2\E \otimes \cL^\vee) \oplus (\Lambda^4\E \otimes (\cL^2)^\vee) \cong \O_{\P^1}^4 \oplus \O_{\P^1}(-1)^4,\\
            & g_{3*}\O_{C_3} \cong \Z_3\cong \O_{\P^1}\oplus (\Lambda^4\E \otimes (\cL^2)^\vee) \cong \O_{\P^1}\oplus \O_{\P^1}(-1).
        \end{split}
    \end{equation*}
    Using the convention (\ref{basis}) for local sections $v_1, v_2, e_3, e_4$ of $\E$, we can describe the corresponding local section for each direct summand of $\B_{m,0}$ and the table of its local sections is given as below.
    \begin{equation}\label{tab-ls}
        \begin{array}{c|c|c|c|c|c|c|c|c}
            \text{direct summand} & \O_{\P^1} & \O_{\P^1} & \O_{\P^1} & \O_{\P^1} & \O_{\P^1}(-1) & \O_{\P^1}(-1) & \O_{\P^1}(-1) & \O_{\P^1}(-1)\\
        \hline
            \text{local section} & v_1v_2 & v_1e_3 & v_1e_4 & v_2v_1 & v_2e_3 & v_2e_4 & v_2v_1e_3e_4 & v_1v_2e_3e_4
        \end{array}
    \end{equation}
    Lemma \ref{3to2} implies that
    \begin{equation*}
        g_{3*}h_* \sEnd (\cV_3)\cong \O_{\P^1}^4 \oplus \O_{\P^1}(-1)^4
    \end{equation*}
    and one can use (\ref{tab-ls}) to check that $g_{3*}\beta$ is an isomorphism. Since $g_3$ is finite, $\beta$ is also an isomorphism. 
\end{proof}

Recall the definition of $h_m$ from (\ref{hm}). From Lemma \ref{doublecover} (ii) and Proposition \ref{3rdcase} (i) we get
\begin{equation}\label{end-v}
    \B_{m,0}\cong g_{m*}\wt{\B_{m,0}} \cong h_{m*}\sEnd(\cV_m), \quad m=1,2,3.
\end{equation}
Combining functors (\ref{fmphi}) (\ref{gmeq}) and Proposition \ref{3rdcase} (ii), we have for $m=1,2,3$ the fully faithful embedding
\begin{equation}\label{fmpsi}
    \begin{tikzcd}
        \Psi_{m,l}\colon \D(L_m) \arrow{r}{-\otimes \cV_m^\vee}[swap]{\cong} & \D(L_m, \sEnd(\cV_m)) \arrow{r}{h_{m*}}[swap]{\cong} & \D(\P^1, \B_{m,0}) \arrow[hook]{r}{\Phi_{m,l}} & \D(Y_m).
    \end{tikzcd}
\end{equation}

Recall maps $p_m\colon Y_m\to \P^1$, $\pi\colon \P(\E)\to \P^1$ and $i_m\colon Y_m\hookrightarrow \P(\E)$ ($p_m = \pi \circ i_m$). We know from Proposition \ref{setup} (iii) that $\O_{\P^1}\subset \E$ is the isotropic sub line bundle corresponding to the smooth section $E=\P(\O_{\P^1})$ of $p_m$. Its associated spinor sheaf $\cS_E$ is a rank $2$ vector bundle constructed by the short exact sequence
\begin{equation}\label{spinor-e}
    0\to \O_{\P(\E)/S}(-1)\otimes \pi^*\cI_0\xrightarrow{\delta_m} \pi^* \cI_1\to i_{m*}\cS_E\to 0
\end{equation}
where $\cI_0\subset \B_{m,0}$ and $\cI_1\subset \B_{m,1}$ are right modules over $\B_{m,0}$ generated by $\O_{\P^1}$, respectively.

\begin{lemma}\label{se}
    Write $\cS_E(-1)$ for $\cS_E \otimes \O_{Y_m/\P^1}(-1)$. We have

    (i) $\Psi_{m,l}(\O_{L_m}) \cong \cS_E(-1) \otimes p_m^*\O_{\P^1}(l)$;

    (ii) $\det(\cS_E^\vee) \cong \O_{Y_m/\P^1}(-1)\otimes p_m^*\O_{\P^1}(2)$ and $\cS_E^\vee \cong \cS_E(-1) \otimes p_m^*\O_{\P^1}(2)$.
\end{lemma}
\begin{proof}
    (i) The isomorphisms in (\ref{end-v}) imply that $h_{m*}(\cV_m^\vee)$ is a right module over $\B_{m,0}$. Since $\cV_m=\O_{L_m}\oplus \cL_m$, we have $\cV_m^\vee\subset \sEnd(\cV_m)$ and thus $h_{m*}(\cV_m^\vee)\subset \B_{m,0}$.  We claim $h_{m*}(\cV_m^\vee)\cong \cI_0$.

    To prove the claim we adopt the notations in Example \ref{qfex} and use the convention (\ref{basis}). Note that $v_1$ is the local section of $\O_{\P^1}\subset \E$. The claim would follow if we can prove that $h_{m*}(\cV_m^\vee)$ is a rank $4$ vector bundle with local sections $v_1v_2, v_1e_3, v_1e_4, v_1v_2e_3e_4$. For $m=1,2$, we have $L_m=C_m, g_m=h_m$ and we get from (\ref{matrix}) that
    \begin{equation*}
        \sHom(\O_{L_m}, \O_{L_m}) \quad \text{has a local section} \quad v_1v_2.
    \end{equation*}
    From (\ref{actionbycenter}) we get that 
    \begin{equation*}
        h_{m*}\sHom(\O_{L_m}, \O_{L_m}) \quad \text{is a rank $2$ vector bundle with local sections} \quad v_1v_2, v_1v_2e_3e_4.
    \end{equation*}
    Recall from (\ref{vb12}) that
    \begin{equation*}
        h_{m*}\sHom(\cL_m, \O_{C_m}) \quad \text{is a rank $2$ vector bundle with local sections} \quad v_1e_3, v_1e_4.
    \end{equation*}
    This concludes the cases of $m=1,2$. For $m=3$, we get from (\ref{end-sec})(\ref{beta-map}) that
    \begin{equation*}
        h_*\sEnd(\cV_3, \O_{L_3})\cong \O_N\oplus\O_{N_1}\oplus \O_{N_2} \quad \text{has local sections} \quad v_1v_2, v_1e_3, v_3e_4
    \end{equation*}
    where $N=C_3$ and $N_1, N_2$ are irreducible components of $N$. Recall $h_3=g_3\circ h$. We deduce from (\ref{actionbycenter}) that $h_{3*}(\cV_3^\vee)$ is a rank $4$ vector bundle with local sections $v_1v_2, v_1e_3, v_1e_4, v_1v_2e_3e_4$.

    Note that $\cI_0\otimes_{\B_{m,0}} \B_{m,1} \cong \cI_1$ where $\cI_1\subset\B_{m,1}$ is the right module over $\B_{m,0}$ generated by $\O_{\P^1}\subset \E$. By (\ref{fmphi})(\ref{kerses}) $\Phi_{m,l}(\cI_0) = p_m^*\cI_0 \otimes^{\LL}_{p_m^*\B_{m,0}} K_{m,l}$ fits into the short exact sequence
    \begin{equation*}
    0\to \O_{\P(\E)/\P^1}(-2)\otimes \pi^*(\cI_0\otimes \O_{\P^1}(l)) \xrightarrow{\delta_m} \O_{\P(\E)/\P^1}(-1)\otimes \pi^*(\cI_1 \otimes \O_{\P^1}(l)) \to i_{m*} \Phi_{m,l}(\cI_0) \to 0. 
    \end{equation*}
    Comparing it with (\ref{spinor-e}), we get $\Psi_{m,l}(\O_{L_m}) \cong \Phi_{m,l}(\cI_0) \cong \cS_E(-1) \otimes p_m^*\O_{\P^1}(l)$. 

    (ii) We can apply Proposition \ref{spinorsheaves} (ii) by letting $\cV=\E=\O_{\P^1}\oplus \O_{\P^1}(-1)^3$, $\sL=\cL=\O_{\P^1}(-1)$, $\cN=\O_{\P^1}$ and $F=E$. Then there are isomorphisms
    \begin{equation*}
        \cS_E^\vee \cong \cS_E(-1) \otimes p_m^*(\O_{\P^1}^\vee \otimes \det(\E^\vee)\otimes \cL) \cong \cS_E(-1) \otimes p_m^*\O_{\P^1}(2).
    \end{equation*}
    This gives $\det(\cS_E^\vee)^2\cong \O_{Y_m/\P^1}(-2)\otimes p_m^*\O_{\P^1}(4)$. On the other hand, the proof of Proposition \ref{spinorsheaves} (i) implies that $\det(\cS_E^\vee)$ restricted to each fiber of $p_m\colon Y_m\to \P^1$ is $\O(-1)$. Thus, 
    \begin{equation*}
        \det(\cS_E^\vee) \cong \O_{Y_m/\P^1}(-1)\otimes p_m^* M
    \end{equation*}
    for some line bundle $M$ on $\P^1$ and it follows that $M\cong\O_{\P^1}(2)$.
\end{proof}

\begin{cor}\label{ymsodnew}
    Write $\cS_E(-1)$ for $\cS_E\otimes \O_{Y_m/\P^1}(-1)$. For $m=1,2,3$ there is an admissible semiorthogonal decomposition
    \begin{equation*}
        \D(Y_m) = \langle  \D(R_{m-1}),  \O_{Y_m/\P^1}(-1)\otimes p_m^*\O_{\P^1}(1), \cS_E(-1) \otimes p_m^*\O_{\P^1}(1), p_m^*\O_{\P^1}(-1), \O_{Y_m},  \O_{Y_m/\P^1}(1) \rangle
    \end{equation*}
    where $R_0=\kk$ and $R_1, R_2$ are the path algebras defined by (\ref{rn}).
\end{cor}
\begin{proof}
    The semiorthogonal decomposition (\ref{ymsod}) when $l=1$ reads as
    \begin{equation*}
        \D(Y_m)= \langle \Phi_{m,1} (\D(\P^1, \B_{m,0})), p_m^*\O_{\P^1}(-1), \O_{Y_m},  \O_{Y_m/\P^1}(1), \O_{Y_m/\P^1}(1)\otimes p_m^*\O_{\P^1}(1)\rangle.    
    \end{equation*}
    By the construction of $\Psi_{m,1}$ from (\ref{fmpsi}), we have $\Psi_{m, 1}(\D(L_m))\cong \Phi_{m,1} (\D(\P^1, \B_{m,0}))$. Hence, the above semiorthogonal decomposition is equivalent to
    \begin{equation*}
        \begin{split}
            & \langle \Psi_{m, 1}(\D(L_m)), p_m^*\O_{\P^1}(-1), \O_{Y_m},  \O_{Y_m/\P^1}(1), \O_{Y_m/\P^1}(1)\otimes p_m^*\O_{\P^1}(1) \rangle\\
            \cong & \langle  \D(R_{m-1}), \cS_E(-1) \otimes p_m^*\O_{\P^1}(1), p_m^*\O_{\P^1}(-1), \O_{Y_m},  \O_{Y_m/\P^1}(1), \O_{Y_m/\P^1}(1)\otimes p_m^*\O_{\P^1}(1) \rangle\\
            \cong & \langle  \O_{Y_m}(-1)\otimes p_m^*\O_{\P^1}(1), \D(R_{m-1}), \cS_E(-1) \otimes p_m^*\O_{\P^1}(1), p_m^*\O_{\P^1}(-1), \O_{Y_m},  \O_{Y_m/\P^1}(1) \rangle\\
            \cong & \langle  \D(R_{m-1}),  \O_{Y_m/\P^1}(-1)\otimes p_m^*\O_{\P^1}(1), \cS_E(-1) \otimes p_m^*\O_{\P^1}(1), p_m^*\O_{\P^1}(-1), \O_{Y_m},  \O_{Y_m/\P^1}(1) \rangle.
        \end{split}
    \end{equation*}
    The first equivalence above is due to Proposition \ref{chainp1} and Lemma \ref{se} (i). For the second equivalence we use Lemma \ref{adsodlem} with $X=Y_m$ and $S=\Spec \kk$. By Proposition \ref{setup} (i)(iv), $Y_m$ is Gorenstein projective and $\omega_{Y_m}\cong \O_{Y_m/\P^1}(-2)$. Then the second equivalence is obtained by applying the functor $-\otimes \omega_{Y_m}\cong -\otimes \O_{Y_m/\P^1}(-2)$ to $\O_{Y_m/\P^1}(1)\otimes p_m^*\O_{\P^1}(1)$. The final equivalence is obtained by the left mutation of $\D(R_{m-1})$ through $\O_{Y_m/\P^1}(-1)\otimes p_m^*\O_{\P^1}(1)$. Again by Lemma \ref{adsodlem} all SODs above are admissible when they exist. The final mutation exists because the SOD before the mutation is admissible, and the mutation induces equivalence of categories before and after the mutation.
\end{proof}

Recall maps $p_m\colon Y_m\to \P^1$ and $f_m\colon Y_m\to X_m$. In the next lemma we check that Proposition \ref{kawaprop} can be applied to the map $f_m$.
\begin{lemma}\label{dbr2}
    Write $M_1=\O(D_1)=\O_{Y_m}(-1)\otimes p_m^*\O_{\P^1}(1)$ and $M_2=\O(D_2)=p_m^*\O_{\P^1}(-1)$. Then

    (1) $D_1.E=1$ and $D_2.E=-1$;

    (2) $\{N_i\coloneqq f_{m*}M_i\}_{i=1,2}$ is a simple collection of sheaves; i.e., $\dim \Hom(N_j, N_k)=\delta_{jk}$ for $1\leqslant j, k\leqslant 2$;

    (3) $H^p(X_m, R^0f_{m*}\O(D_i-D_j))=0$ for all $p>0$ and $1\leqslant i, j\leqslant 2$;

    (4) the triangulated subcategory $\langle N_i\rangle_{i=1}^2$ of $\D(X_m)$ generated by $N_1, N_2$ is equivalent to $\D(R_1)$ where $R_1$ is the path algebra defined by (\ref{rn}).
\end{lemma}
\begin{proof}
    (1) Note $\O_{Y_m}(-1)|_E\cong \O_E$ by Proposition \ref{setup} (iv). Then (1) follows from this and that $E$ is a section of $p_m$.

    (2) Let $I$ be the ideal $I_{E/Y_m}$ of $E\subset Y_m$ and let $E_n$ be the subscheme defined by $I^n, n\geqslant 1$. Then
    \begin{equation*}
        I/I^2\cong N_{E/Y_m}^\vee \cong \O_{\P^1}(1)^2, \quad I^n/I^{n+1} \cong \Sym^n(I/I^2) \cong \O_{\P^1}(n)^{n+1}, \quad n\geqslant 2.
    \end{equation*}
    Let $\mu_n\colon E_n\to Y_m, n\geqslant 1$ be the inclusion maps. By the theorem on formal functions we have
    \begin{equation*}
        \widehat{R^pf_{m*}(M)} \cong \varprojlim H^p(E_n, L_0\mu_n^*M)
    \end{equation*}
    for all $M\in \Coh(Y_m)$. Consider for $n\geqslant 1$ the exact sequences
    \begin{equation*}
        0\to I^n/I^{n+1} \to \O_{E_{n+1}}\to \O_{E_n} \to 0.
    \end{equation*}
    Since $M_1|_E\cong \O_{\P^1}(1)$ and $M_2|_E \cong \O_{\P^1}(-1)$, we have $M_i|_{E_n}$ are extensions of sheaves $\O_{\P^1}(l)$ for $l \geqslant -1$. This implies that $H^p(E_n, M_i|_{E_n})=0$ for $p>0, i=1,2$ and all $n$. Thus, $R^pf_{m*}(M_i)=0, p>0$ and $N_i$ is a sheaf for $i=1,2$. 

    Next we show that $\{N_1, N_2\}$ is a simple collection. Let $U=Y_m-E$. Let $j_U\colon U\hookrightarrow Y_m$ and $i_U=f_m\circ j_U\colon U\hookrightarrow X_m$ be open embeddings. Since the codimension of $U\subset Y_m$ is $2$, we have $M_i\cong R^0j_{U*} (M_i|_U)$. Thus,
    \begin{equation*}
        \begin{split}
            \sHom(N_i, N_j) & \cong \sHom(R^0f_{m*}(M_i), R^0f_{m*}(M_j))\\
            & \cong \sHom(R^0f_{m*}R^0j_{U*}(M_i|_U), R^0f_{m*}R^0j_{U*}(M_j|_U))\\
            & \cong \sHom(R^0i_{U*}(M_i|_U), R^0i_{U*}(M_j|_U))\\
            & \cong R^0i_{U*}\sHom(M_i|_U, M_j|_U).
        \end{split}
    \end{equation*}
    Taking $H^0(X_m, -)$ on both sides above, we get
    \begin{equation*}
        \Hom(N_i, N_j)\cong \Hom(M_i|_U, M_j|_U)\cong \Hom(M_i, M_j).
    \end{equation*}
    Since 
    \begin{equation*}
        \begin{split}
            & p_{m*}\O_{Y_m}\cong \O_{\P^1},\quad p_{m*}\O(D_1-D_2)=0,\\
            & p_{m*}\O(D_2-D_1)\cong \E^\vee\otimes \O_{\P^1}(-2) \cong \O_{\P^1}(-2)\oplus \O_{\P^1}(-1)^3,
        \end{split}
    \end{equation*} 
    we have 
    \begin{equation}\label{hp}
        h^p(Y_m, \O(D_i-D_j))=h^p(\P^1, p_{m*}\O(D_i-D_j))=\left\{
        \begin{array}{ll}
            \delta_{ij}, & p=0\\
            0, & (i,j)=(1,2), \forall p\\
            0, &  (i,j)=(2,1), p\neq 1\\
            1, & (i,j)=(2,1), p=1
        \end{array}
        \right..
    \end{equation}
    Hence, $\dim \Hom(N_i, N_j)=\delta_{ij}$.

    (3) Since $\O_E(D_1-D_2)\cong \O_{\P^1}(2)$, we have that $\O(D_1-D_2)|_{E_n}$ is the extension of sheaves $\O_{\P^1}(l)$ for $l\geqslant 2$. The arguments using the theorem on formal functions in (2) imply that
    \begin{equation*}
        R^qf_{m*}\O(D_1-D_2)=0, q>0.
    \end{equation*}
    Similarly, we have $\O(D_1-D_2)|_{E_n}$ is the extension of $\O_{\P^1}(-2)$ with sheaves $\O_{\P^1}(l)$ for $l\geqslant -1$. This implies that 
    \begin{equation*}
        R^1f_{m*}\O(D_2-D_1)\cong \kk,
    \end{equation*}
    which is supported on the node $a_1\in X_m$. We make use of the Leray spectral sequence
    \begin{equation*}
        E^{p,q}_2= H^p(X_m, R^qf_{m*}\O(D_i-D_j)) \Rightarrow E^{p+q}= H^{p+q}(Y_m, \O(D_i-D_j)).
    \end{equation*}
    Since $R^{\geqslant 2}f_{m*}=0$, $R^1f_{m*}$ is supported on the node $a_1\in X_m$ and $\dim(X_m)=3$, we have $E^{p,q}_2=0$ unless $0\leqslant p \leqslant 3, q=0$ or $(p,q)=(0,1)$. In this case there is an exact sequence
    \begin{equation*}
        0\to E^{1,0}_2\to E^1\xrightarrow{\xi} E^{0,1}_2\to E^{2,0}_2\to E^2
    \end{equation*}
    and $E^{3,0}_2 \cong E^3$. Equations (\ref{hp}) give $E^2=E^3=0$ for all $i,j$, and $E^1=E^{0,1}_2=0$ if $(i,j)\neq (2,1)$ and $E^1 \cong E^{0,1}_2 \cong \kk$ if $(i,j) = (2,1)$. We deduce that $\xi\colon E^1 \to E^{0,1}_2$ is an isomorphism for all $i,j$. Thus, $E^{p,0}_2=0$ for all $p>0$ and $i,j$.

    (4) It is a consequence of (1)-(3) and Proposition \ref{kawaprop}.
\end{proof}

For the smooth section $E$ of $p_m\colon Y_m\to \P^1$, Proposition \ref{spinorsheaves} (ii) gives us the Koszul resolution
\begin{equation*}
    0\to \det(\cS_E^\vee) \to \cS_E^\vee \to \O_{Y_m} \to \O_E\to 0.
\end{equation*}
By Lemma \ref{se} (ii) we get
\begin{equation}\label{kozulresfib}
    0\to \O_{Y_m/\P^1}(-1)\otimes p_m^*\O_{\P^1}(1)\to \cS_E(-1) \otimes p_m^*\O_{\P^1}(1)\to p_m^*\O_{\P^1}(-1)\to \O_E(-1)\to 0.
\end{equation}
\begin{thm}\label{mainthm}
    Let $X_m$ be the quintic del Pezzo threefolds with $m$ nodes for $m=1,2,3$. Then there is an admissible semiorthogonal decomposition
    \begin{equation*}
        \D(X_m)=\langle \D(R_{m-1}), \D(R_1), \O_{X_m}, \O_{X_m}(1)\rangle
    \end{equation*}
    where $R_0=\kk$ and $R_1, R_2$ are the path algebras defined by (\ref{rn}).
\end{thm}
\begin{proof}
    From Proposition \ref{pushforward} we have that $f_{m*}\colon \D(Y_m)\to \D(X_m)$ is a Verdier quotient with $\ker(f_{m*})=\langle \O_E(-1)\rangle$. Let $\cT$ be the triangulated subcategory generated by
    \begin{equation*}
        \O_{Y_m}(-1)\otimes \O_{\P^1}(1),\; \cS_E(-1) \otimes \O_{\P^1}(1),\; p_m^*\O_{\P^1}(-1).
    \end{equation*}
    Then Corollary \ref{ymsodnew} reads as
    \begin{equation}\label{dbym}
        \D(Y_m)=\langle \D(R_{m-1}), \cT, \O_{Y_m}, \O_{Y_m}(1)\rangle.
    \end{equation}
    We deduce from the exact sequence (\ref{kozulresfib}) that $\O_E(-1)\in \cT$. Observe that in the SOD (\ref{dbym}) all components except for $\D(R_{m-1})$ are contained in $\DD^{\perf}(Y_m)$. Moreover, $Y_m$ is Gorenstein and $\omega_{Y_m}|_E\cong \O_E$ by Proposition \ref{setup} (i)(iv). Then Proposition \ref{indsod} implies that we have the induced admissible semiorthogonal decomposition
    \begin{equation*}
        \begin{split}
            \D(X_m) &=\langle f_{m*}\D(R_{m-1}), f_{m*}(\cT), f_{m*}(\O_{Y_m}), f_{m*}(\O_{Y_m}(1)) \rangle\\
            & \cong \langle f_{m*}\D(R_{m-1}), f_{m*}(\cT), \O_{X_m}, \O_{X_m}(1)\rangle.
        \end{split}
    \end{equation*}
    Since $\D(R_{m-1})\cap \ker(f_{m*})=\emptyset$, we have $f_{m*}\D(R_{m-1})\cong \D(R_{m-1})$. Since $f_{m*}(\cT)$ is generated by $f_{m*}(\O_{Y_m}(-1)\otimes p_m^*\O_{\P^1}(1))$ and $f_{m*}(p_m^*\O_{\P^1}(-1))$, Lemma \ref{dbr2} (4) implies that $f_{m*}(\cT)\cong \D(R_1)$.
\end{proof}
\bibliography{ndp5.bib}
\bibliographystyle{abbrv}
\end{document}